\documentclass[12pt,centertags,oneside]{amsart}

\pdfoutput=1

\usepackage{tikz}
\usetikzlibrary{shapes.geometric}
\usetikzlibrary {angles,calc,quotes}

\usepackage{circuitikz}
 \usepackage{caption}
 
\usepackage{amsmath,amstext,amsthm,amscd,typearea,hyperref,stmaryrd,mathabx}
\usepackage{amssymb}
\usepackage{a4wide}
\usepackage[mathscr]{eucal}
\usepackage{mathrsfs}
\usepackage{typearea}
\usepackage{charter}
\usepackage{pdfsync}
\usepackage[a4paper,width=16.2cm,top=3cm,bottom=3cm,lines=50]{geometry}

\numberwithin{equation}{section}

\newtheorem{theorem}{Theorem}[section]
\newtheorem{definition}[theorem]{Definition}
\newtheorem{proposition}[theorem]{Proposition}
\newtheorem{corollary}[theorem]{Corollary}
\newtheorem{lemma}[theorem]{Lemma}
\newtheorem{remark}[theorem]{Remark}

\newtheorem{example}[theorem]{Example}

\newcommand{\cali}[1]{\mathscr{#1}}
 
\newcommand{\Nor}{\mathop{\mathrm{Nor}}\nolimits}

\newcommand{\Leb}{{\rm Leb}}

\newcommand{\supp}{{\rm supp}}

\newcommand{\dist}{{\rm dist}}

\newcommand{\ddc}{{dd^c}}

\newcommand{\dc}{{d^c}}

\renewcommand{\Re}{{\rm Re}}
\renewcommand{\Im}{{\rm Im}}

\newcommand{\Cc}{\cali{C}}
\newcommand{\Dc}{\cali{D}}

\newcommand{\Fc}{\cali{F}}

\newcommand{\Lc}{\cali{L}}

\newcommand{\Sc}{\cali{S}}

\newcommand{\Ic}{\cali{I}}

\newcommand{\C}{\mathbb{C}}
\newcommand{\D}{\mathbb{D}}

\renewcommand{\H}{\mathbb{H}}
\newcommand{\N}{\mathbb{N}}
\newcommand{\Z}{\mathbb{Z}}

\newcommand{\R}{\mathbb{R}}
\newcommand{\T}{\mathbb{T}}
\newcommand{\B}{\mathbb{B}}
\newcommand{\K}{\mathbb{K}}
\newcommand{\U}{\mathbb{U}}

\renewcommand{\S}{\mathbb{S}}
\renewcommand{\P}{\mathbb{P}}
\newcommand{\X}{\mathbb{X}}
 \newcommand{\I}{\mathbb{I}}

\title[]{Singular holomorphic foliations by curves.  III: Zero Lelong numbers }

\author{Vi{\^e}t-Anh Nguy{\^e}n}
\address{Universit\'e de Lille, 
Laboratoire de math\'ematiques Paul Painlev\'e, 
CNRS U.M.R. 8524,  
59655 Villeneuve d'Ascq Cedex, 
France. {\tt  https://pro.univ-lille.fr/viet-anh-nguyen/}}
\email{Viet-Anh.Nguyen@univ-lille.fr}

\date{March 25, 2022}

\begin{document}


\begin{abstract}
Let $\Fc$ be a  holomorphic  foliation by curves defined in a neighborhood of $0$ in $\C^n$ ($n\geq 2$)
 having  $0$ as  a weakly hyperbolic  singularity.
 Let $T$ be a positive harmonic 
 current  directed  by $\Fc$  which 
    does not give  mass  to any of the $n$ coordinate invariant hyperplanes $\{z_j=0\}$ for $1\leq j\leq n.$ 
Then we  show that the Lelong number  of $T$ at $0$ vanishes.
Moreover, an application of this local result in the  global context is   given.  
We   discuss  also the relation  between 
several basic notions such as directed  positive harmonic currents, directed positive $\ddc$-closed  currents, Lelong numbers etc.   in the  framework of 
singular  holomorphic foliations. 
\end{abstract}

\maketitle

\medskip\medskip


 \medskip



\section{Introduction} \label{S:Intro}

 The  aim of this  article is  twofold.
 Its  first (but not main) purpose is to  revisit the basis of several  fundamental notions in the theory of singular holomorphic  foliations such as:  directed positive harmonic currents, 
 directed  positive $\ddc$-closed currents, Lelong numbers etc.
 The second (but  main) purpose  of the  article is
 to prove  the following local result and apply it to several  global contexts.

\begin{theorem}\label{T:main} {\rm (Main Theorem)} 
 Let $\Fc:=(\D^n,\Lc,\{0\})$ with  $n\geq 2,$    be  
a     holomorphic  foliation, which is  defined  on the unit polydisc $\D^n$ of $\C^n$ 
and 
which is  associated  to the linear vector field 
$$\Phi(z) = \sum_{j=1}^n\lambda_j z_j {\partial\over \partial z_j} ,\qquad  z=(z_1,\ldots,z_n),$$ where  $\lambda_j$ are all nonzero complex numbers  and there  are some $1\leq l\not=k\leq n$
   with $\lambda_k/\lambda_l\not\in \R.$  
  Let  $T$ be a  positive  harmonic current   directed by $\Fc$   which 
    does not give  mass  to any of the $n$ coordinate invariant hyperplanes $\{z_j=0\}$.  Then   the   Lelong  number of $T$ at the origin $0:=(0,\ldots,0)$  vanishes.  
 \end{theorem}
 Note that  the hypothesis on the linear vector field means that $0$ is  an isolated weakly  hyperbolic  singularity of $\Fc$
 and  $\Fc$ has no  other singularity.
 

It  is  natural  to  ask   how and to what extent   the value of  the current $T$ near  the union  $$\mathcal Z:=\D^n\cap \bigcup\limits_{j=1}^n\{z_j=0\}$$  of the  coordinate invariant hyperplanes on $\D^n$
affects   the conclusion   of the Main Theorem. The next result, which gives also a  stronger version of  the Main Theorem, answers  this question. 

\begin{theorem}\label{T:main_bis}
Let $\Fc$ be the  foliation  as in   Theorem \ref{T:main} and  $\widecheck\Fc$  the  restriction of   $\Fc$  to  $\D^n\setminus \mathcal Z.$
Let  $T$ be a  positive  harmonic current on $\D^n\setminus \mathcal Z$  directed by $\widecheck \Fc$ such that
the mass of $T$ on $\D^n\setminus (r\D)^n$  is  finite  for some $r\in (0,1).$ Here $(r\D)^n$ denotes the polydisc
of polyradius $r$ in $\C^n.$ 
Then  the following assertions hold:
\begin{enumerate} 
\item The mass of $T$  on $\D^n$  is  finite.
\item The Lelong number of $T$ at every point  of $\mathcal Z $  vanishes.
\end{enumerate}
\end{theorem}

Combining  Theorem \ref{T:main} and  some results of  Forn{\ae}ss--Sibony  \cite{FornaessSibony05,FornaessSibony10},  the  following global  picture  is obtained for
 directed positive harmonic  currents living on compact complex  manifolds.
 \begin{theorem}\label{T:main_global}
Let $\Fc=(X,\Lc,E)$    be  
a     singular  holomorphic  foliation   with      the set of  singularities $E$ in a    compact complex manifold  $X.$ 
Assume that
\begin{enumerate}
 \item   there is  no invariant analytic  curve;
\item all the singularities  are  hyperbolic linearizable;
\item  there is   no  non-constant holomorphic map $\C\to X$ such that
out of $E$   the image of $\C$ is locally contained in  a leaf.
\end{enumerate}
Then, for every  positive  harmonic  current $T$  directed by $\Fc,$   $T$  is  diffuse  and the   Lelong  number of $T$ vanishes everywhere   in $X.$  
 \end{theorem}

The  above  theorem  and   results by  Brunella  \cite{Brunella},   Jouanolou \cite{Jouanolou} and Lins Neto-Soares \cite{NetoSoares}, give us the following corollary.
It can be   applied  to  every generic  foliation in $\P^n$ with a given degree $d>1.$
  
\begin{corollary}\label{C:main_2}
 Let $\Fc=(\P^n,\Lc,E)$    be  
a   singular  foliation by Riemann surfaces on the complex projective space $\P^n$  with $n\geq 2.$ Assume that
all the singularities  are hyperbolic and  that  $\Fc$ has no invariant algebraic  curve. 
Then for every  positive harmonic  current $T$  directed by $\Fc,$  $T$ is  diffuse and  the   Lelong  number of $T$ vanishes everywhere   in $\P^n.$ 
\end{corollary}

It is  worthy noting that  the above results  generalize   our previous  work  \cite{NguyenVietAnh18a}  to all dimensions.   
The last  two decades  witness many important advances  in the  theory of holomorphic foliations by curves on ambient  complex  surfaces  emphasizing on 
 singular  holomorphic foliations.    The reader is invited to consult
 the surveys  \cite{DinhSibony20,FornaessSibony08,NguyenVietAnh18c,NguyenVietAnh20b} for systematic  expositions.
The  present work is  motivated by these exciting  developments. Our distant goal is  trying  to understand  the theory  in the  general  case of   higher  dimensions $n>2$.
Therefore,  one of the first  steps  should be  to investigate local  situations  near the singularities of  the  foliation in  question.

\medskip

This  point of view  seems to be  fruitful  in dimension $n=2.$ 
Indeed, the  work of  Forn{\ae}ss--Sibony  \cite{FornaessSibony05, FornaessSibony08, FornaessSibony10}   initiates
the local study of positive harmonic measure near a  hyperbolic singularity  for this dimension.
This  study is  an important tool  for  further developments of the  theory, see e.g.  \cite{DinhNguyenSibony18,DinhSibony18,FornaessSibony10}.
A  typical  feature in dimension $n=2$ is  that  the phase spaces are not only  simple, but also    essentially unique modulo a translation and a rotation,  see  Figure \ref{fig:M1}.
This makes the  analysis in dimension  $n=2$  feasible.
Roughly speaking,  a {\it phase space} $\Pi_x$ is a  domain in $\C$ which parametrizes  the part  of the leaf $L_x$  inside  the unit polydisc $\D^n,$  where $x$ is a point in $\D^n\setminus \{0\}$ (see Section \ref{S:Geometry} for more details).
In fact, in dimension $n=2,$  the (unique) phase space is a sector and  the behavior of its Poisson kernel determines the mass-repartition  of the positive harmonic  currents
near  the  hyperbolic singularity.
In \cite{NguyenVietAnh18a} the author  revisits  this  question and obtains  a complete behavior of the Poisson kernel of  the phase space.
This  result  plays a vital role in  the  author's several subsequent developments  \cite{NguyenVietAnh18b,NguyenVietAnh18d} when  he   studies the  Lyapunov exponent
of singular holomorphic  foliations  living on complex surfaces.

\begin{figure}
\centering
\begin{minipage}{.2\textwidth}
\begin{tikzpicture}
\draw (0,0) node [below] {$A$}; 
\node at (0,0) {$\bullet$};
   \draw  [->] (0,0)--(75:4cm) node [above] {$y$};
     \draw  [->] (0,0)--(25:4cm) node [right] {$x$};
     \draw[->] (25:1cm)  arc  (25:75:1cm);
     \node at (50:1.25cm)  {$\theta$};
\end{tikzpicture}
\end{minipage}
\hspace{3cm}
\begin{minipage}{.2\textwidth}
\begin{tikzpicture}
\draw (0,0) node [below] {$O$}; 
\node at (0,0) {$\bullet$};
   \draw  [->] (0,0)--(50:4cm) node [above] {$y$};
     \draw  [->] (0,0)--(0:4cm) node [right] {$x$};
     \draw[->] (0:1cm)  arc  (0:50:1cm);
     \node at (25:1.25cm)  {$\theta$};
\end{tikzpicture}
\end{minipage}
\caption{On the left: the phase space of a foliation with a hyperbolic  singularity in dimension $n=2:$ 
a sector  with  central angle  $\theta.$ On the right:   the phase space  is unique (i.e. with vertex at the origin $O$ and with $Ox$ the real axis)
modulo a translation and a rotation.} \label{fig:M1}
\end{figure}
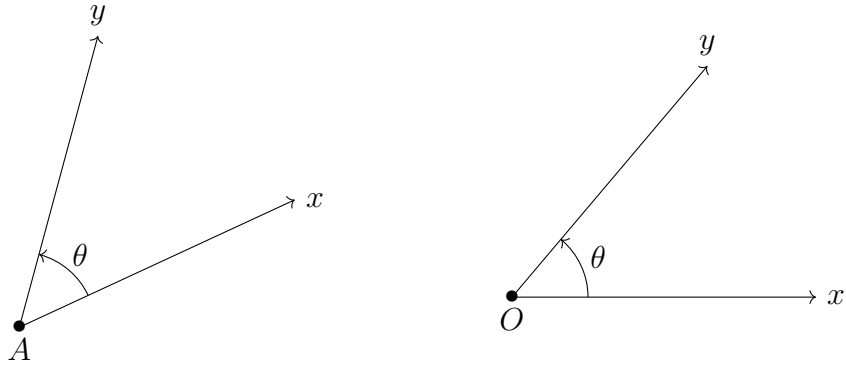


Unfortunately,  in higher  dimension $n>2,$ the geometry of phase spaces $\Pi_x$ is not  simple as well as  not  unique  any more. Figures \ref{fig:M2} and  \ref{fig:M3}
describe all possible  phase spaces in dimension $n=3.$
As  we will see in Section \ref{S:Geometry}, $\Pi_x$ is a convex $m(x)$-gon, where the integer $m(x)$  varies  between $2$ and $ n.$ Moreover, $\Pi_x$ may be  bounded or unbounded.
The  Poisson kernels of these  phase spaces are difficult to  study. Although they are all conformally equivalent to  the unit disc $\D$ by a   Schwarz-Christoffel mapping
(see e.g. \cite{DT}),  this  tentative attempt  turns out to be  not   realistic. Indeed,  we have, in principle, Schwarz-Christoffel  formula in order to compute this mapping and hence 
the Poisson kernel of $\Pi_x.$ 
But this  formula  is  useful  only when we understand very well the shape of the  phase space in question, and  even if it is the case, when $n$ is  large,
the formula  only gives us a small information  on the Poisson kernel of the domain $\Pi_x$ near its boundary. 
However,  the shape of $\Pi_x$  changes drastically in terms of  $x\in\D^n\setminus \{0\},$ in particular, when $x$ approaches the coordinate hyperplanes.
Therefore, this  formula  alone  does not work.

\begin{figure}
\centering
\begin{minipage}{.4\textwidth}
\begin{tikzpicture}
\draw (0,0)  node [below] {$A$};
\node at (0,0) {$\bullet$};
\draw[->] (0:1cm)  arc  (0: atan(4/3)):1cm);
\draw (5,0)  node [below] {$B$};
\node at (5,0) {$\bullet$};
\draw (4,0)  arc  (180: 180-atan(3/4)):1cm);
     \draw (3.9,0)  arc  (180: 180-atan(3/4)):1.1cm);
      \draw (3.95,0)  arc  (180: 180-atan(3/4)):1.05cm);
     \draw (0,0)--(5,0); 
     \draw (0,0)--(1.8,2.4); 
     \draw  (1.8,2.4) node [above] {$C$};
     \node at (1.8,2.4) {$\bullet$};
     \draw (1.5,2)  arc  (180+atan(4/3): 270+atan(4/3)):0.5cm);
      \draw (1.47,1.96)  arc  (180+atan(4/3): 270+atan(4/3)):0.55cm);
     \draw (5,0)--(1.8,2.4);

\end{tikzpicture}
 \end{minipage}
\begin{minipage}{.4\textwidth}
\begin{tikzpicture}
\draw (0,0) node [below] {$A$}; 
\node at (0,0) {$\bullet$};
     \draw  [->] (0,0)--(5,0) node [right] {$x$};
      \draw  [->] (0,0)--(1.8,2.4) node [above] {$y$};
      \draw[->] (0:1cm)  arc  (0: atan(4/3)):1cm);
\end{tikzpicture}
\end{minipage}

\begin{minipage}{.4\textwidth}
\begin{tikzpicture}
\draw (5,0) node [below] {$B$}; 
\node at (5,0) {$\bullet$};
\draw (4,0)  arc  (180: 180-atan(3/4)):1cm);
     \draw (3.9,0)  arc  (180: 180-atan(3/4)):1.1cm);
      \draw (3.95,0)  arc  (180: 180-atan(3/4)):1.05cm);
     \draw  [->] (5,0)--(0,0) node [left] {$x$};
      \draw  [->] (5,0)--(1.8,2.4) node [above] {$y$};
\end{tikzpicture}
\end{minipage}
\begin{minipage}{.4\textwidth}
\begin{tikzpicture}
\draw (1.8,2.4) node [above] {$C$}; 
\node at (1.8,2.4) {$\bullet$};
\draw (1.5,2)  arc  (180+atan(4/3): 270+atan(4/3)):0.5cm);
      \draw (1.47,1.96)  arc  (180+atan(4/3): 270+atan(4/3)):0.55cm);
     \draw  [->] (1.8,2.4)--(0,0) node [left] {$x$};
      \draw  [->] (1.8,2.4)--(5,0) node [right] {$y$};
\end{tikzpicture}
\end{minipage}

\caption{The phase spaces of a  foliation with a hyperbolic  singularity  in dimension $n=3:$ the first (the triangle $ABC$) is unique modulo the composition of  a translation and a dilation, whereas the  remaining three sectors are    unique modulo a translation.} \label{fig:M2}
\end{figure}
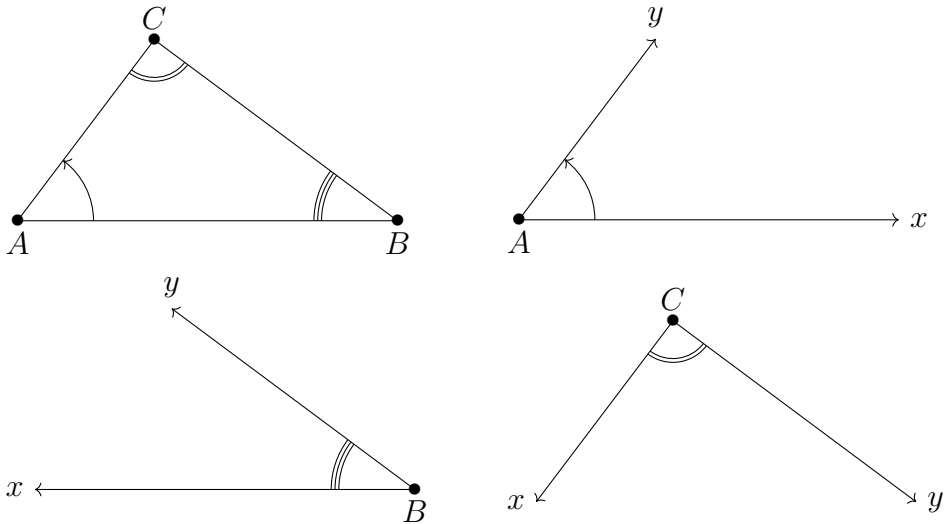

Our main  idea  is  to  use   the comparison principle  of Poisson kernels and to  combine it with  the complete  behavior  of Poisson kernel in dimension $2$ provided by \cite{NguyenVietAnh18a} and
Schwarz-Christoffel  formula. The comparison principle of Poisson kernels is a well-known technique in Harmonic Analysis where it often applies to  bounded smooth  domains.
In the present work, the principle applies to  phase spaces which are, in general,  neither smooth nor  bounded.
We do not obtain a complete behavior of the Poisson kernel of the phase spaces  as in dimension $2$, but instead we get  their asymptotic behavior which suffices for our purpose.
We hope that the techniques developed in this work will be  useful in many other problems.

\medskip

The article  is  organized as  follows. In   Section \ref{S:Background}     we  strengthen the basis of  the  theory of singular  holomorphic  foliations and  set up the  background 
of the article.  So  this  section  fulfills  the  first  purpose   of this  article.
The  rest of the article is  devoted  to the second (and main) purpose.  More  specifically,  
Section \ref{S:Geometry} studies the geometry of a  singular flow box.
Here, we will  see that  the phase spaces as well as  other    related  objects  encountered  in  dimension $n>2$   are much more  complicated  to  understand than  those  in dimension $n=2.$
 Our main  estimates   are  developed  in   Section  \ref {S:Main-estimates} which  are the core of the  work.
The proofs  of  Theorem \ref{T:main} and  Theorem   \ref{T:main_bis}  will be  provided   in Section \ref{S:Main-Theorem}.
The  proofs of Theorem  \ref{T:main_global}  and  Corollary   \ref{C:main_2} will be given in  Section \ref{S:Other-results}.
The  article  is concluded  with some  remarks and  open questions.
  
  \smallskip

  \noindent {\bf Notation.} Throughout  the paper, 
  \begin{itemize} \item $\D$ denote   the unit disc in $\C,$  and for $r>0,$ $r\D$ denotes the disc of center $0$  and of radius $r.$  
   \item For an open set $\Omega\subset \C,$  $\partial \Omega$ denotes the topological boundary of $\Omega$ and $P_\Omega$  denotes its Poisson kernel.
\item  $\Leb_1$ (resp.  $\Leb_2$) denotes the one-dimensional  (resp. two-dimensional) Lebesgue measure.
 \item The letters $c,$ $c',$ $c'',$ $c_0,$  $c_1,$ $c_2$ etc. denote  positive constants, not necessarily the same at each  occurrence. 
 \item The notation $\gtrsim$ and $\lesssim$ means inequalities  up to  a  multiplicative constant, whereas  we  write  $\approx$ when  both inequalities  are satisfied.

  \end{itemize}
\smallskip

\noindent
{\bf Acknowledgments. }  
The  author  acknowledges   support by the Labex CEMPI (ANR-11-LABX-0007-01) and by the project QuaSiDy  (ANR-21-CE40-0016).
The paper was partially prepared 
during the visit of the  author at the Vietnam  Institute for Advanced Study in Mathematics (VIASM). He would like to express his gratitude to this organization for hospitality and  for  financial support.

 \section{Background}\label{S:Background}
    In this  section we undertake the  first (not main)  task of this  work. Namely, we revisit the background of the  theory  of singular holomorphic foliations  emphasizing the relations between some  basic notions such as
      directed positive harmonic currents, 
 directed  positive $\ddc$-closed currents etc.  The  survey
 \cite{NguyenVietAnh20b} gives a unified treatment in a more general context of a lamination  which is holomorphically immersed  in a complex  manifold.   
 See also the survey \cite{FornaessSibony08} for the  original discussion of these notions.     
 
    \subsection{Positive $\ddc$-closed currents and Lelong number}
 Let $X$ be a  complex manifold of dimension $n.$
 We fix an atlas
of $X$ which is locally finite. Up to reducing slightly the charts, we
can assume that the local coordinate system associated to
each chart is defined on a neighbourhood of the closure of this chart. For $0\leq p,q\leq k$ and $l\in\N$, denote by  $\Dc^{p,q}_l(X)$ the space of  $(p,q)$-forms  of
class $\Cc^l$ with compact support in $X,$ and  $\Dc^{p,q}(X)$ their
intersection for $l\in\N.$ If $\alpha$ is a  $(p,q)$-form on $X$, denote by
$\|\alpha\|_{\Cc^l}$ the sum of the $\Cc^l$-norms of the coefficients
of $\alpha$ in the local coordinates. These norms induce a
topology   on $\Dc^{p,q}_l(X)$ and $\Dc^{p,q}(X)$. In particular, a
sequence $\alpha_j$ converges to $\alpha$  in  $\Dc^{p,q}(X)$ if
these forms are supported in a fixed compact set and if
$\|\alpha_j-\alpha\|_{\Cc^l}\rightarrow 0$ for every $l$. 
 
 A {\it $(p,q)$-current} on  $X$ (or equivalently, a {\it current  of bidegree $(p,q),$} or equivalently,   a {\it current  of bidimension  $(n-p,n-q)$}) is
 a  continuous linear form $T$  on $\Dc^{n-p,n-q}(X)$ with values in $\C.$

A $(p,p)$-form on  $X$  is {\it
  positive} if it can be written at every point as a combination with
positive coefficients of forms of type
$$i\alpha_1\wedge\overline\alpha_1\wedge\ldots\wedge
i\alpha_p\wedge\overline\alpha_p$$
where the $\alpha_j$ are $(1,0)$-forms. A $(p,p)$-current or a $(p,p)$-form $T$ on $X$ is
{\it weakly positive} if $T\wedge\varphi$ is a positive measure for
any smooth positive $(n-p,n-p)$-form $\varphi$. A $(p,p)$-current $T$
is {\it positive} if $T\wedge\varphi$ is a positive measure for
any smooth weakly positive $(n-p,n-p)$-form $\varphi$. 
If $X$ is given with a Hermitian metric $\beta$ and $T$ is  a positive  $(p,p)$-current on $X,$
$T\wedge \beta^{n-p}$ is a positive measure on
$X$. The mass of  $T\wedge \beta^{n-p}$
on a measurable set $A$ is denoted by $\|T\|_A$ and is called {\it the mass of $T$ on $A$}.
{\it The mass} $\|T\|$ of $T$ is the total mass of  $T\wedge \beta^{n-p}$ on $X.$


 A $(p,p)$-current on  $X$  is {\it
closed} if  $d T=0$ in the  weak sense (namely,  $T(d \alpha)=0$ for every  test form $\alpha\in \Dc^{n-p,n-p-1}(X)\oplus\Dc^{n-p-1,n-p}(X)$). 
 A $(p,p)$-current on  $X$  is {\it
$\ddc$-closed} if  $\ddc T=0$ in the  weak sense (namely,  $T(\ddc \alpha)=0$ for every  form $\alpha\in \Dc^{n-p-1,n-p-1}(X)$).

Let $T$ be  a positive $\ddc$-closed current on  $X.$  A fundamental theorem of Skoda \cite{Skoda} says that  the Lelong number  of $T$ at  a  point $x\in X,$ defined by
 \begin{equation}\label{e:Lelong}
 \nu(T,x):= \lim_{r\to 0+}{1\over  \pi^{n-p} r^{2(n-p)}}\int_{\B(x,r)} T\wedge ( \ddc\|z\|^2)^{n-p}.
 \end{equation}
 always  exists and is finite non-negative.
Here, we identify, via a local coordinate $z,$ a neighborhood
of $x$ in $X$ to an open neighborhood  of $0$ in  $\C^n,$    and $\B(x,r)$ is thus identified with 
the Euclidean ball in $\C^n$ with center $0$ and radius $r.$  
 In fact,   Siu  \cite{Siu} (see also \cite{Demailly}) shows that   when $T$  is  a  positive  closed  current, the  Lelong
number $\nu(T,x)$
  is  independent of 
the choice of local  coordinates  near $x.$  The same  result  for  positive  $\ddc$-closed currents  is  proved  by  
    Alessandrini--Bassanelli
  \cite{AlessandriniBassanelli96}.

The  next  simple  result    allows  for extending  positive $\ddc$-closed  currents of bidimension $(1,1)$    through isolated points.

\begin{theorem} \label{T:extension_1}{\rm  (Dinh-Nguyen-Sibony \cite[Lemma 2.5]{DinhNguyenSibony12},  Forn{\ae}ss-Sibony-Wold \cite[Lemma 17]{FornaessSibonyWold})}
Let $T$ be a positive current of bidimension $(1,1)$ with compact support on a complex manifold $X$. Assume that $\ddc T$ is a negative measure on $X\setminus E$ where $E$ is a finite set. 
Then  $T$ is a positive  $\ddc$-closed  current on $X.$
\end{theorem}

When   the support of $T$  is not compact, we  only have the following  local mass finiteness.

\begin{theorem} \label{T:extension_2}{\rm  (Alessandrini--Bassanelli
  \cite[Main Theorem 5.6]{AlessandriniBassanelli93})}
Let $T$ be a positive $\ddc$-closed  current of bidimension $(1,1)$ outside a single point $x$  on a complex manifold $X$.  Then  the mass of $T$  is  finite near $x.$
\end{theorem}
 \subsection{Singular holomorphic foliations and  hyperbolic singularities}
     
Let $X$ be a complex   manifold of dimension $n.$  A {\it   holomorphic foliation (by Riemann surfaces, or equivalently,  by curves)}   $\Fc=(X,\Lc)$ on $X$  is  the  data of  a {\it foliation  atlas} with foliated charts 
$$\Phi_p:\U_p\rightarrow \B_p\times \T_p.$$
Here, $\T_p$ and $\B_p$  are  domains in $\C^{n-1}$ and in $\C$ respectively,    $\U_p$ is  a  domain in 
$X,$ and  
$\Phi_p$ is  biholomorphic,  and  all the changes of coordinates $\Phi_p\circ\Phi_q^{-1}$ are of the form
$$x=(y,t)\mapsto x'=(y',t'), \quad y'=\Psi(y,t),\quad t'=\Lambda(t).$$

The open set $\U_p$ is called a {\it flow
  box} and the Riemann surface $\Phi_p^{-1}\{t=c\}$ in $\U_p$ with $c\in\T_p$ is a {\it
  plaque}. The property of the above coordinate changes insures that
the plaques in different flow boxes are compatible in the intersection of
the boxes. Two plaques are {\it adjacent} if they have non-empty intersection.
 
A {\it leaf} $L$ is a minimal connected subset of $X$ such
that if $L$ intersects a plaque, it contains that plaque. So a leaf $L$
is a  Riemann surface  immersed in $X$ which is a
union of plaques. A leaf through a point $x$ of this foliation is often denoted by $L_x.$
  A {\it transversal} is  a complex submanifold  of codimension 1 in  $X$  which is  transverse to  the leaves  of $\Fc.$

A {\it    holomorphic foliation   with singularities} is  the data
$\Fc=(X,\Lc,E)$, where $X$ is a complex   manifold, $E$ a closed
subset of $X$ and $(X\setminus E,\Lc)$ is a  holomorphic  foliation. Each point in $E$ is  said to be  a {\it  singular point,} and
$E$ is said to be {\it the  set of singularities} of the foliation. We always  assume that $\overline{X\setminus E}=X$, see 
e.g. \cite{DinhNguyenSibony12,FornaessSibony08, NguyenVietAnh18c,NguyenVietAnh20b}  for more details.

  Consider a   holomorphic foliation    $\Fc=(X,\Lc,E)$ and an isolated point $x$ of $E.$  We say that a
 $x$ is {\it linearizable} if 
there is a (local) holomorphic coordinates system of $X$ on an open neighborhood $\U$ of $x$ on which
$\Fc|_{\U_x}=(\U_x,\Lc|_{\U_x},\{x\})$ is identified with  $\widehat\Fc:=(\D^n,\widehat\Lc, \{0\})$ and 
 the
leaves of $\widehat\Fc$  are, under this  identification, integral curves of a  linear vector field
$$\Phi(z) = \sum_{j=1}^n\lambda_j z_j {\partial\over \partial z_j},\qquad  z=(z_1,\ldots,z_n),$$
where  $\lambda_j$ are some nonzero complex numbers.
Such  a neighborhood  $\U_x$ is  called 
a {\it singular flow box} of  $x.$ $\widehat\Fc=(\D^n,\widehat\Lc, \{0\})$ is called  a {\it local model} of the  linearizable  singularity $x.$ 

\begin{definition}\label{e:singular-points} \rm 
 Let  $x\in E$  a
linearizable singular point of $\Fc$ as  above.
\begin{itemize}
\item We say that is $x$ is {\it weakly hyperbolic}  if  
 there  are {\bf some} $1\leq j\not=k\leq n$
   with $\lambda_j/\lambda_k\not\in \R.$  
\item 
   We  say that $x$  is
{\it hyperbolic} if
 $\lambda_j/\lambda_k\not\in \R$ for 
{\bf all} $1\leq j\not=k\leq n.$ 
\end{itemize}
\end{definition}
\begin{remark}
 \rm
 When $n=2,$   weakly hyperbolic  singular point $=$ hyperbolic singular point. But for $n>2,$ the weakly hyperbolic  singularity is
 strictly  weaker  than the  hyperbolic  singularity.  
\end{remark}

 \subsection{Positive harmonic currents and  directed positive $\ddc$-closed currents}
 
 Let  $\Fc=(X,\Lc,E)$ be a  singular holomorphic  foliation  on a complex manifold $X$ of dimension $n.$
 
 A   {\it (directed) $(p,q)$-form} on $\Fc$ can be seen on the flow box
$\U\simeq \B\times\T$ as
a  $(p,q)$-form  on $\B$ depending on the parameter $t\in \T$. 
For  $0\leq p,q\leq 1$,  denote by  $\Dc^{p,q}_l(\Fc)$ the space of 
 $(p,q)$-form $f$ with compact support in $X\setminus E$ satisfying the following property: 
$f$ restricted to each flow box $\U\simeq \B\times\T$ is a 
 $(p,q)$-form of class $\Cc^l$ on the plaques whose coefficients and
all their derivatives up to order $l$ depend continuously on the
plaque. The norm $\|\cdot\|_{\Cc^l}$ on
this space is defined as in the case of  manifold using a locally finite
atlas of $\Fc$. We also define   $\Dc^{p,q}(\Fc)$  as the intersection of
  $\Dc^{p,q}_l(\Fc)$ for $l\geq 0$.
  In particular, a
sequence $f_j$ converges to $f$  in  $\Dc^{p,q}(\Fc)$ if
these forms are supported in a fixed compact set of $X\setminus E$  and if
$\|f_j-f\|_{\Cc^l}\rightarrow 0$ for every $l$.
A {\it (directed) current of bidegree $(p,q)$ ({\rm or equivalently,}  of bidimension $(1-p,1-q)$)} 
on $\Fc$ is a continuous
linear form on the space $\Dc^{1-p,1-q}(\Fc)$ 
with values in $\C$. 
  We often write for short $\Dc(\Fc)$ instead of  $\Dc^{0,0}(\Fc).$
  
    A form $\alpha \in \Dc^{1,1}(\Fc)$ is  said to be {\it positive} if its restriction to every plaque
 is  a  positive $(1,1)$-form in the  usual  sense.
 \begin{definition}\label{D:Directed_hamonic_currents} 
\rm  Let $T$ be a directed current of bidimension $(1,1)$ on $\Fc.$ 
    
  $\bullet$  $T$ is  said to be  {\it positive} if  $T(f)\geq 0$ for all positive forms $f\in \Dc^{1,1}(\Fc).$

$\bullet$ $T$ is  said to be  {\it   closed}   if  $d T=0$ in the  weak sense (namely,  $T(d f)=0$ for all directed forms  $f\in  \Dc^1(\Fc)$).

$\bullet$ $T$ is  said to be  {\it   harmonic}  if  $\ddc T=0$ in the  weak sense (namely,  $T(\ddc f)=0$ for all functions  $f\in  \Dc(\Fc)$).
\end{definition}

  We have the  following decomposition.
  \begin{proposition}{\rm  (see  e.g. \cite[Propositions  2.1, 2.2 and  2.3]{DinhNguyenSibony12})  }\label{P:decomposition}
   Let $T$ be  a directed harmonic current on $\Fc.$  Let $\U\simeq \B\times \T$ be  a flow  box which is relatively compact in $X.$
   \begin{enumerate}
    \item {\rm  (Existence)}
Then, there is a positive Radon measure $\mu$ on $\T$ and for $\mu$-almost every $t\in\T,$ there is a  harmonic   function $h_t$ on $\B$
   such that  if $K$  is compact in $\B,$  the integral $\int_K \|h_t\|_{L^1(K)} d\mu(t)$ is  finite and
   $$
   T(\alpha)=\int_\T\big( \int_\B  h_t(y)\alpha(y,t)  \big)d\mu(t)
   $$
   for every  form $\alpha\in\Dc^{1,1}(\Fc)$ compactly  supported  on $\U.$ 
   \item {\rm (Uniqueness)} If  $\mu'$ and $h'_t$  are associated  with  another  decomposition  of $T$ in $\U,$ then  there is  a  measurable  function $\theta>0$
   on a measurable  subset $\T'\subset \T$   such that $h_t=0$  for $\mu$-almost every $t\not\in\T',$  $h'_t=0$  for $\mu'$-almost every $t\not\in\T',$
   and $h_t=\theta(t) h'_t$ for $\nu$ and $\nu'$-almost every $t\in\T'.$  
   \item
    If  moreover, $T$ is  positive, then    for $\mu$-almost every $t\in\T,$ the  harmonic   function $h_t$ is  positive on $\B.$
   \item  If  moreover, $T$ is  closed, then    for $\mu$-almost every $t\in\T,$ the  harmonic   function $h_t$ is  constant on $\B.$
 \end{enumerate}
 \end{proposition}
  
 \begin{definition} \label{D:diffuse}\rm 
  A directed positive harmonic current $T$ on $\Fc=(X,\Lc,E)$  is said to be  {\it diffuse}  if   for any  decomposition of $T$ in any flow box  $\U\simeq \B\times \T$
  as in Proposition \ref{P:decomposition}, the  measure $\nu$ has no mass on each single point of the transversal $\T.$
 \end{definition}

For a complex  manifold $M,$ let $\Dc^{1,1}(M)$ denote the space of smooth $(1,1)$-forms $\alpha$  compactly supported in $M$ endowed with the  semi-norms $\| \cdot\|_{\Cc^l(M_k)},$   where $l\in\N$ and $(M_k)_{k=0}^\infty$ is an increasing sequence of relatively compact open subsets of $M$ such that $M=\bigcup_{k=0}^\infty M_k.$   
For  $x\in X\setminus E,$ let $j_x:\  L_x\hookrightarrow X$   be   the canonical injective immersion   from $L_x$ into $ X. $  
The aggregate of  the  pull-back  via $j_x$  of  each test form $\alpha\in  \Dc^{1,1}(X\setminus E)$  defines
a form in  $\Dc^{1,1}(\Fc)$ denoted by $j^*\alpha.$
So  we obtain  a canonical restriction  map $$j^*:\   \Dc^{1,1}(X\setminus E)\to \Dc^{1,1}(\Fc)\qquad\text{given by}\qquad  \alpha\mapsto  j^*\alpha.$$
We see  easily that  the image $\Ic$  of $j^*$  is dense in  $\Dc^{1,1}(\Fc).$ 

The original  notions  of  directed positive $\ddc$-closed currents for  singular  holomorphic foliations (resp. for singular laminations which are holomorphically  immersed in a complex  manifold)
with  a small set of singularities were  introduced by  Berndtsson-Sibony \cite{BerndtssonSibony} 
(resp.  by Forn{\ae}ss-Sibony \cite{FornaessSibony05,FornaessSibony08}).
In  \cite{NguyenVietAnh20b}  we give   another   notion of directed positive $\ddc$-closed currents for singular Riemann surface  laminations which are holomorphically immersed in a complex  manifold.
  Our  notion coincides with the previous ones  when  the lamination is $\Cc^2$-transversally smooth, this   condition is  automatically  fulfilled when $\Fc$ is a singular holomorphic foliation.  The  advantage  of our (slighly improved) notion is that
  it is  relevant even  when the set of singularities is  not small.
  We  recall our definition in the   present context  of  singular holomorphic  foliations. 

 \begin{definition}\label{D:Directed_hamonic_currents_with_sing}   \rm 
Let  $\Fc=(X,\Lc,E)$   be  a   singular  holomorphic foliation.
A {\it directed positive $\ddc$-closed  current}  (resp.  a  {\it directed  positive closed current}) on $\Fc$ is  a  positive $\ddc$-closed current $T$  (resp.   a  positive closed current $T$)
of bidimension $(1,1)$ on $X$ such that 
the following properties (i)-(ii)  are satisfied:
\begin{itemize}  
 \item [(i)]  $T$  does not give mass to $E,$  i.e. the mass $\|T\|_E$  of $T$ on $E$  is  zero;
\item[(ii)]
 $T$   is  a  directed positive harmonic  current (resp.  a  directed positive closed  current) on $\Fc$  in the sense of Definition \ref{D:Directed_hamonic_currents}.
\end{itemize}
Moreover, we say that $T$ is  {\it diffuse}  if  it is  diffuse in the sense of Definition \ref{D:diffuse} as   a  directed positive harmonic  current (resp.  a  directed positive closed  current) on $\Fc.$ 
  \end{definition}

  \begin{remark}\label{R:two-notions-in-comparison} \rm 
   When $E=\varnothing,$ property (i) of Definition \ref{D:Directed_hamonic_currents_with_sing} is trivially satisfied and property (ii) says that   a directed positive $\ddc$-closed  current
   $T$  may be regarded as a positive harmonic  current.  The converse  statement (when $E=\varnothing$) is  also  true, see Subsection \ref{SS:Lelong-harmonic-currents} below,
  \end{remark}

  \begin{remark}\rm  \label{R:explanation}  
  Property (ii) of Definition \ref{D:Directed_hamonic_currents_with_sing}
 means  the  following two properties (ii-a)-(ii-b):
\begin{itemize}\item[(ii-a)] $\langle T,\alpha\rangle=\langle T,\beta\rangle$ for $\alpha,\beta\in  \Dc^{1,1}(X\setminus E)  $ such  that $j^*\alpha=j^*\beta;$
 so the current $$\widetilde T:\ \Ic\to\C\quad\text{given by}\quad \langle \widetilde T, j^*\alpha\rangle:=  \langle T,\alpha\rangle\quad \text{for}\quad \alpha\in \Dc^{1,1}(X\setminus E),  $$
  is  well-defined;
 \item[(ii-b)]  the current $\widetilde T$ defined in  (ii-a) can be uniquely   extended from $\Ic$ to   $\Dc^{1,1}(\Fc)$ by continuity  (as $\Ic$ is  dense in  $\Dc^{1,1}(\Fc)$)  to a current
 $ \widehat T $ of order zero, and $\widehat T$ is 
 a   directed positive harmonic  current (resp.   directed positive closed  current) on $\Fc$  in the sense of Definition \ref{D:Directed_hamonic_currents}.
 \end{itemize}
 Property (ii-b)  holds  automatically since  $T$ is a positive $\ddc$-closed current (resp.  positive closed current)  on $X.$ So property (ii) is  equivalent to  the  single property (ii-a).
 If there is no confusion,   we often denote $\widetilde T$ and $\widehat T$ simply by $T.$
  \end{remark}

  As an immediate  consequence  of Proposition \ref{P:decomposition} and  Definition   \ref{D:Directed_hamonic_currents_with_sing}, we obtain  the following characterization  of  directed  positive
  $\ddc$-closed  currents  directed by a singular  holomorphic foliation.
  \begin{proposition}\label{P:decomposition-bis}
   Let  $\Fc=(X,\Lc,E)$   be  a   singular  holomorphic foliation and   $T$   a  positive  $\ddc$-closed  current (resp.  a   positive  closed  current) of bidimension $(1,1)$  on $X.$
   Then  $T$ is a  directed positive $\ddc$-closed  current  (resp.  a  directed  positive closed current) on $\Fc$ if and only  if
the following properties (i)-(ii)'  are satisfied:
\begin{itemize}  
 \item [(i)]  $T$  does not give mass to $E,$  i.e. the mass $\|T\|_E$  of $T$ on $E$  is  zero;
\item[(ii)']
   For any  flow box $\U\simeq \B\times \T$  which is relatively compact in $X,$
   there is a positive Radon measure $\mu$ on $\T$ and for $\mu$-almost every $t\in\T,$ there is a  harmonic   function $h_t$ on $\B$
   such that  if $K$  is compact in $\B,$  the integral $\int_K \|h_t\|_{L^1(K)} d\mu(t)$ is  finite and
   $$
   T(\alpha)=\int_\T\big( \int_\B  h_t(y)\alpha(y,t)  \big)d\mu(t)
   $$
   for every  form $\alpha\in\Dc^{1,1}(X)$ compactly  supported  on $\U.$ 
  \end{itemize}
  Moreover,   if  $T$ is  positive $\ddc$-closed  current   (resp.   positive closed current), then    for $\mu$-almost every $t\in\T,$ the  harmonic   function $h_t$ is  positive on $\B$
    (resp.  the  harmonic   function $h_t$ is  constant on $\B$).
  
  \end{proposition}

  
  \subsection{Lelong number of positive  harmonic currents}\label{SS:Lelong-harmonic-currents}
  
  Let $\Fc=(X,\Lc,E)$ be a  singular holomorphic foliation and 
   $T$   a  positive  harmonic  current.  By  Proposition  \ref{P:decomposition} (1),  we  see that  
 $$ T'(\alpha):=  T(j^*\alpha)\qquad\text{for}\qquad \alpha\in \Dc^{1,1}(X\setminus E)  
 $$
 is   a positive  $\ddc$-closed  current of bidimension $(1,1)$ on  $X\setminus E.$  In fact, Remark \ref{R:two-notions-in-comparison}
 says  that  outside the singularities,  directed positive  $\ddc$-closed currents  coincide with  positive harmonic currents.
 In other words, $T'$ may be canonically  identified with $T$  outside $E.$

 {\it The mass of  $T$
on a (measurable) set $A,$} denoted by $\|T\|_A,$  is   the mass of $T$ (as a  positive $\ddc$-closed current) on $A\setminus E.$
{\it The mass} $\|T\|$ of $T$ is the total mass of  $T$ on $X.$ 

 For every $x\in X\setminus E,$    we can define  
 \begin{equation}\label{e:Lelong_bis}
 \nu(T,x):=\nu(T',x), 
 \end{equation}
where the right-hand side is  given by  \eqref{e:Lelong} with $p=n-1.$

  For  $x\in E,$  we cannot use  \eqref{e:Lelong}  immediately because $T$ may not be  extended through an open neighborhood of $x$ as  a positive $\ddc$-closed current.  
  However, following  the model formula \eqref{e:Lelong}
  we still define  
  \begin{equation}\label{e:Lelong_bisbis}
 \nu(T,x):= \limsup_{r\to 0+}{1\over  \pi r^2}\int_{\B(x,r)\setminus E} T\wedge ( \ddc\|z\|^2)\in\R^+\cup\{\infty\}.
 \end{equation}
 as the Lelong nummber  of $T$ at $x.$
Here, we identify, through a local coordinate $z,$ a neighborhood
of $x$ in $X$ to an open neighborhood  of $0$ in  $\C^n,$    and $\B(x,r)$ is thus identified with 
the Euclidean ball in $\C^n$ with center $0$ and radius $r.$  
 A priori, the  Lelong
number $\nu(T,x)$ may depend on  
the choice of   local  coordinates  near $x.$ However, it is  easy to check that
if $\nu(T,x)$ is  equal to either $0$ or $\infty,$ then it is  independent of local  coordinates  near $x.$

  \begin{example}
   \rm
   Let 
   $ \Fc:=(\D^2,\Lc,\{0\})$ be a singular holomorphic foliation associated to a linear vector field in $\D^2.$
   Consider  the current $T(z):=(-\log |z_2|)\cdot [z_1=0]$  for $z=(z_1,z_2)\in\D^2\setminus \{0\}.$
   Here $[z_1=0]$ is the current of integration on the complex line $\{z_1=0\}.$
   We can check that $T$ is a positive harmonic current for $\Fc,$ but $\nu(T,0)=\infty.$
    In fact, the simple extension $\widehat T$  of $T$  through $\{0\}$  satisfies $\ddc \widehat T=-\delta_0,$ where $\delta_0$ is  the Dirac mass at the origin.
    So  $\widehat T$ is  positive, but not  $\ddc$-closed.
  \end{example}

  \subsection{Green function and Poisson kernel}
  
 \begin{definition} \label{D:Green} \rm Let $\Omega \subset \C$ be either a (not necessarily bounded) convex polygon  or a  $\Cc^2$-smooth bounded domain. The {\it diagonal} $\Delta$ of $\Omega$
 is   the  set $\{(\zeta,\zeta):\  \zeta\in\Omega\}\subset\Omega\times\Omega.$  A function $G : (\Omega \times \overline\Omega) \setminus \Delta \to\R$ is {\it the Green's
function}  on $\Omega$ if:
\begin{enumerate}
\item  for each fixed $\zeta \in \Omega$ the function $ G (\zeta, \xi) +{1\over2\pi}\log|\xi - \zeta|$ is harmonic as a
function of $\xi\in\Omega$ (even at the point $\zeta$);
\item $ G(\zeta, \xi)|_{\xi\in\partial\Omega}= 0$ for each fixed $\zeta\in\Omega.$
\end{enumerate}
\end{definition}
Let $\Nor=\Nor_\Omega$ represent the unit outward normal vector field on $\partial \Omega.$  The  following classical  result gives the {\it Poisson kernel} of bounded smooth  domains in $\C$
(see e.g.  \cite{Krantz}).
\begin{proposition}\label{P:Poisson-classic}
 Let $\Omega\subset \C$ be   a  $\Cc^2$-piecewise smooth bounded domain. Then
 \begin{enumerate}
  \item There is a unique  Green  function $G$ on $\Omega.$   
  \item 
Let   the Poisson kernel  on $\Omega$ be the function
 \begin{equation*}
  P(\zeta,\xi)=P_\Omega(\zeta,\xi):= -\Nor_\xi G(\zeta,\xi)\qquad\text{for}\qquad  \zeta\in\Omega,\ \xi\in\partial \Omega.
 \end{equation*}
 Then  the following  two assertions  hold:
 \begin{itemize}
\item[(2-i) ]If $u\in\Cc(\overline\Omega)$ is  harmonic  on $\Omega,$ then
\begin{equation*}
 u(\zeta)=\int_{\partial\Omega} P(\zeta,\xi)u(\xi)d\Leb_1(\xi)\qquad\text{for}\qquad  \zeta\in\Omega.
\end{equation*}
Here, $d\Leb_1$ is the $1$-dimensional Lebesgue measure on $\partial \Omega.$
\item[(2-ii)]  $P(\cdot,\xi)$ is a positive harmonic function   on $\Omega$  when $\xi\in\partial\Omega$ is  fixed.
\end{itemize}
 \end{enumerate}
\end{proposition}

\section{Geometry of singular flow boxes}\label{S:Geometry}
    
Let $\U$ be  an open  neigborhood of the closed  $n$-unit polydisc  $\overline\D^n.$    
Consider the foliation $\Fc=(\U,\Lc,\{0\})$ which is the restriction to $\U$ of the foliation associated to the vector field
\begin{equation}\label{e:local-model}\Phi(z)=\sum_{j=1}^n \lambda_jz_j{\partial\over \partial z_j},\qquad  z\in\C^n,
 \end{equation}
with $\lambda_j\in\C^*$.  The foliation is singular at the origin. 
We often  call $\Fc$ a {\it  local model} of   a linearizable  singularity.

    Write $\lambda_j=a_j+ib_j$ with $a_j,b_j\in\R$.  
For $x=(x_1,\ldots,x_n)\in \U\setminus\{0\}$, define the holomorphic map $\varphi_x:\C\rightarrow\C^n\setminus\{0\}$ by
\begin{equation}\label{e:varphi_x}
\varphi_x(\zeta):=\Big(x_1e^{\lambda_1\zeta},\ldots,x_ne^{\lambda_n\zeta}\Big)\quad \mbox{for}\quad \zeta\in\C.
\end{equation}
It is easy to see that $\varphi_x(\C)$ is the integral curve of $\Phi$ which contains
$\varphi_x(0)=x$. 
The leaf of $\Fc$ through $x$ is given by $  L_x:=\varphi_x(\C)\cap \U$.

Write $\zeta=u+iv$ with $u,v\in\R$. Fix $x=(x_1,\ldots,x_n)\in \U\setminus\{0\}.$
Let $ j=1,\ldots,k.$    If $x_j\not=0,$  let $\H_{x,j}$ be the   open half-plane  defined by 
 \begin{equation}\label{e:H_x,j}
\H_{x,j}:=\{\zeta=u+iv\in\C:\ -a_ju+b_jv -\log|x_j|>0\}=  \{ \zeta\in\C:\ |\varphi_x(\zeta)_j|<1    \},
\end{equation}
where $\varphi_x(\zeta)_j$ is  the $j$-th coordinate  of $ \varphi_x(\zeta)\in\C^n.$
If  $x_j=0,$ then   set simply  $ \H_{x,j}:=\C.$
The {\it phase space} of  a  point $x\in \overline\D^n\setminus \{0\}$   is  the domain $\Pi_x:=\varphi_x^{-1}(\D^n)$ in $\C.$   Observe that
\begin{equation}\label{e:phase-space} \Pi_x=\bigcap_{j=1}^n\H_{x,j}.\end{equation}
\begin{remark}\label{R:m(x)-gon} \rm
So, $\Pi_x$ is a convex $m(x)$-gon which is not necessarily
bounded, where $0\leq m(x)\leq n$ is  an integer depending only on $x.$ 
Moreover, for $x,x'\in \overline{\D}^n\setminus \{0\}$ and $1\leq j\leq n$ such that $x_j\not=0,$ $x'_j\not=0,$
$\partial \H_{x,j}$ is either equal to  or  parallel to $\partial \H_{x',j}.$

When  $\{0\}$ is a  weakly hyperbolic singularity, there are at least  two  edges of  $\Pi_x$  which are not parallel.
In particular  when  $n=2$ and $\{0\}$ is a  weakly hyperbolic singularity, the geometry of phase spaces  are  very simple:  $\Pi_x$ is  a sector  and  $m(x)=1$ and for every $y\in \overline\D^n\setminus \{0\},$ there is a unique  translation ${\mathcal T}_{x,y}$ on the plane
such that
$\Pi_y={\mathcal T}_{x,y} (\Pi_x).$   

 All possible  phase spaces in dimension $n=3$  are illustrated in 
Figures \ref{fig:M2} and  \ref{fig:M3}.
\end{remark}

%

\begin{figure}
\centering
\begin{minipage}{.2\textwidth}
\begin{tikzpicture}
\scalebox{0.8}{
\draw (0,0)  node [below] {$A$};
\draw (4,0)  node [below] {$B$};
\node at (4,0) {$\bullet$};
\draw [dashed] (0,0)--(4,0);
\draw (2,2*1.732)  node [left] {$C$};
\node at (2,2*1.732) {$\bullet$};
\draw  [->] (4,0)--(6,0) node [right] {$x$};
\draw  [->] (2,2*1.732)--(3,3*1.732) node [above] {$y$};
  \draw (4,0)--(2,2*1.732);
     \draw [dashed] (0,0)--(2,2*1.732); 
      \draw[->,dashed] (0:1cm)  arc  (0: 60):1cm);
      }
\end{tikzpicture}
 \end{minipage}
\qquad
\qquad
\qquad
\qquad
\begin{minipage}{.2\textwidth}
\begin{tikzpicture}
\scalebox{0.8}{
\draw (0,0) node [below] {$A$}; 
\node at (0,0) {$\bullet$};
     \draw  [->] (0,0)--(6,0) node [right] {$x$};
      \draw  [->] (0,0)--(3,3*1.732) node [above] {$y$};
      \draw[->] (0:1cm)  arc  (0: 60):1cm);}
\end{tikzpicture}
\end{minipage}

\caption{The phase spaces of a  foliation with a hyperbolic  singularity  in dimension $n=3:$ the first (the unbounded polygon $yCBx$) is unique modulo the composition of  a translation and a dilation, whereas the  second (the  sector $xAy$)
is unique modulo a translation.} \label{fig:M3}
\end{figure}
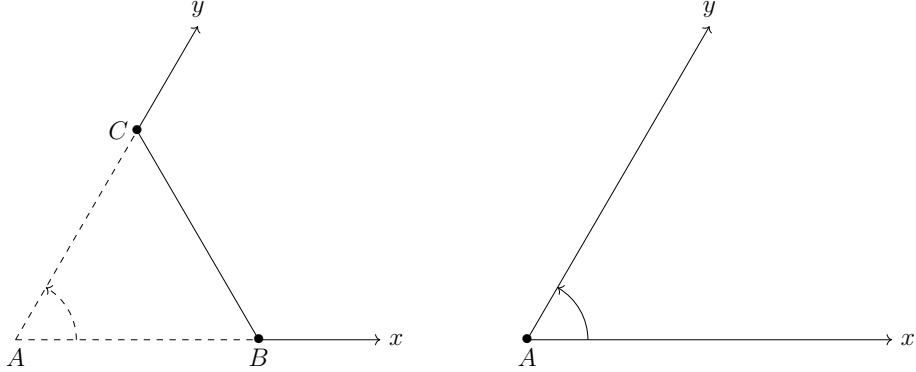

 For  $x\in  \overline\D^n\setminus\{0\} $ consider   
 \begin{equation}\label{e:mathcal-leaf} \mathcal L_x:=\varphi_x(\Pi_x)= L_x\cap  \D^n.
  \end{equation}
 In particular, it is a leaf of the  restriction foliation  $\Fc|_{\D^n },$  where  
 $    \Fc|_{\D^n }:=(\D^n,\Lc|_{\D^n},\{0\})           .$

We use  the standard Euclidean metric $\dist$ on $\C,$  that is,  for $\zeta,\xi\in\C$ and for $A,B\subset\C,$
$$\dist(\zeta,\xi):=|\zeta-\xi|\qquad\text{and}\qquad \dist(A,B):=\inf\limits_{\zeta\in A,\xi\in B}\dist(\zeta,\xi) .$$
Let $\zeta\in\Pi_x$ and  $1\leq j\leq n.$ 
If  $x_j\not=0,$ let $\dist_{x,j}(\zeta)$  be the distance from  $\zeta$ to the real line $\partial\H_{x,j},$
that is,
\begin{equation}\label{e:dist-x_j}
\dist_{x,j}(\zeta):= \dist (\zeta, \partial\H_{x,j} ) =|\lambda_j|^{-1}\big|-a_ju+b_jv -\log|x_j|\big |.
\end{equation}
If $x_j=0,$ set simply $\dist_{x,j}(\zeta):=\infty.$

Note that if  $y=(y_1,\ldots,y_n):=\varphi_x(\zeta),$ then 
\begin{equation}
 \label{e:dist-x_j-vs-y_j}
 \log{|y_j|}:=|\lambda_j|\dist_{x,j}(\zeta).
\end{equation}
Note that $\Pi_x$ contains $0$ since $\varphi_x(0)=x.$  Moreover,  for $x\in \overline\D^n,$
\begin{equation}\label{e:origin-distance}
\dist(0,\partial \Pi_x)=\min \Big\lbrace -{\log|x_1|\over |\lambda_1|},\ldots,  -{\log|x_n|\over |\lambda_n|}  \Big\rbrace. 
\end{equation}
For each $x\in\overline\D^n$ and $\zeta\in\overline\Pi_x,$  there is a  permutation  $\{k_1,\ldots,k_n\}$ of $\{1,\ldots,n\}$ such that
$$
|\lambda_{k_1}|\dist_{x,k_1}(\zeta)\leq |\lambda_{k_2}|\dist_{x,k_2}(\zeta)\leq \ldots\leq|\lambda_{k_n}|\dist_{x,k_n}(\zeta).
$$
Define
\begin{equation}\label{e:dist_x}
 \dist_x(\zeta):= 1+|\lambda_{k_1}|\dist_{x,k_1}(\zeta)\qquad\mbox{and}\qquad 
 \dist^\star_x(\zeta):= 1+|\lambda_{k_n}|\dist_{x,k_n}(\zeta).
\end{equation}
So we obtain  the  following useful estimates:
\begin{equation}\label{e:dist_x-bis}
\min\limits_{1\leq k\leq n}|\lambda_k| \dist(x,\partial \Pi_x)=\dist_x(\zeta)-1 \leq \max\limits_{1\leq k\leq n}|\lambda_k| \dist(x,\partial \Pi_x)\quad\text{and}\quad 
\dist_x(\zeta)\leq  \dist^\star_x(\zeta).
\end{equation}
 The first inequality of  implies  that
 \begin{equation}\label{e:dist_x-bisbis}
\big( \min\limits_{1\leq k\leq n}|\lambda_k|)\cdot \dist(x,\partial \Pi_x)\leq \dist_x(\zeta)-1 \leq \big( \max\limits_{1\leq k\leq n}|\lambda_k|\big)\cdot  \dist(x,\partial \Pi_x).
\end{equation}


Observe that when the ratio $\lambda_i/\lambda_j$ are not all rational and all the coordinates of $x$ do not vanish, 
$\varphi_x:\Pi_x\to\mathcal  L_x$ is bijective and hence $ \mathcal L_x$ is simply connected. Otherwise, when the ratios $\lambda_i/\lambda_j$ are all rational, all the leaves of $\Fc|_{\D^n}$
are closed submanifolds of $\D^k\setminus\{0\}$ and are biholomorphic to annuli.



Let $\I$  denote the quotient
of $(\overline\D\setminus \{0\})^n$ by the equivalence relation $x \sim  y$  if  $\mathcal L_x=\mathcal L_y$  for  $x,y\in (\overline\D\setminus \{0\})^n.$
Let $[x]$  be  the class of  $x$  in this   equivalence  relation.
Let $\pi:\  (\overline\D\setminus \{0\})^n\to \I$ by the canonical  projection  given by  $\pi(x):=[x],$ $x\in  (\overline\D\setminus \{0\})^n.$ 
We  endow $\I$ with  the complex  structure  induced  from  $\C^n.$
\begin{lemma}\label{L:quotient}
Suppose  that $\{0\}$ is  weakly  hyperbolic.  Then:
\begin{enumerate}
\item $\I$ is a  complex  manifold of dimension $n-1.$ 
In particular, when $n=2,$
$\I$ is a complex torus of dimension 1. 
\item  For  every point  $x\in(\overline\D\setminus \{0\})^n,$ there is  a  neighborhood $\U_x$  of $x$   such that  $\U_x$ is  a   flow box $\U_x\simeq \B_x\times \T_x$ 
and that  the restriction $\pi|_{\T_x}$ of 
$\pi$  on the  
transversal $\T_x$    is  biholomorphic onto  its image which is open  subset of $\I.$  In other  words, $\pi$ is locally   biholomorphic in the  transveral  direction.
\end{enumerate}
\end{lemma}
\proof
We  only need to prove the following 

\noindent {\bf  Fact.} {\it For  every $x\in (\overline\D\setminus \{0\})^n$ and every $\epsilon>0,$  there is a  neighborhood $\U_{x,\epsilon}$ of $x$ in  $(\overline\D\setminus \{0\})^n$
such that if $y,z\in \U_{x,\epsilon}$ and $y\sim z,$ then there is $\zeta\in\C$ such that $|\zeta|<\epsilon$ and  
$z=\varphi_y(\zeta),$  where $\varphi_y$ is  defined in \eqref{e:varphi_x}.}

Indeed, taking  the  fact for granted and shrinking   $\U_{x,\epsilon}$ if necessary,  we may assume that  $\U_{x,\epsilon}$ is  a flow box $\B\times \T.$
Consider the restriction $\pi|_{\U_{x,\epsilon}}:\  \U_{x,\epsilon}\to  \I.$  We deduce  from from the fact that  the  restriction $\pi|_{\T}$ of  $\pi|_{\U_{x,\epsilon}}$ to $\T$      is  homemorphic onto  its image which is open  subset of $\I.$
The  assertions of the lemma   follow modulo the above fact.

To  prove  the fact, we infer  from $y\sim z$ that there is  a $\zeta\in\C$ such that $z=\varphi_y(\zeta).$
Since $\{0\}$ is   a weakly hyperbolic  singularity,  there are  $l,l'$  with $1\leq l<l'\leq n$  such that  $t:=\lambda_{l'}/\lambda_l\not\in\R.$
Write  $t=a+ib,$ with $a,b\in\R$ and $b\not=0.$
We choose a neighborhood  $\U_{x,\epsilon}\Subset (\overline\D\setminus \{0\})^n $ of $x$  so that 
$$
{|y_j-x_j|\over |x_j|}\ll  \epsilon\qquad\text{for}\qquad  y=(y_1,\ldots,y_n)\in\U_{x,\epsilon},\,  1\leq j\leq n.
$$
We  infer  from  the above inequalities  and  $y,z\in \U_{x,\epsilon} $ that  for $\epsilon>0$  small enough (depending only on $x$), 
$$|y_j-x_j|\ll \epsilon|x_j|\qquad\text{and}\qquad |y_j-x_j|\ll  \epsilon |x_j|.$$
Hence, 
$$
{|z_j-y_j|\over |y_j|}\ll  \epsilon\qquad\text{for}\qquad   1\leq j\leq n.
$$
Using this,  we infer  from the  equalities $z_j=y_je^{\lambda_j\zeta}$ for $j\in\{l,l'\}$ that
 there are   $k,k'\in\Z$  such that 
$$
|\lambda_{l'}\zeta-2\pi k'|\ll \epsilon\qquad\text{and}\qquad |\lambda_l\zeta-2\pi k|\ll \epsilon.
$$
So
$
| tk-   k'|<\epsilon,
$ which is  equivalent to    $|k(a+ib)-k'|<\epsilon.$ Hence,  $|k'|<|b|^{-1}\epsilon.$

When  $\epsilon<\min(|b|^{-1},1),$  we deduce from $k'\in\Z$ that $k'=0.$
This, combined with the previous  inequality $
|\lambda_{l'}\zeta-2\pi k'|\ll \epsilon,$ yields that $|\zeta|<\epsilon.$
Hence, the fact  follows.

\endproof

\begin{lemma}\label{L:transversal}
By shrinking  $\U$ if necessary,
there is a  Borel  subset $\X\subset(\overline\D\setminus \{0\})^n$  
with  the  following  properties:
\begin{enumerate}
\item For  every $ x =(x_1,\ldots,x_n)\in\X,$ there are two  indexes $1\leq j <l\leq n$ such that $ |x_j|=1$ and $|x_l|=1.$
\item  The restriction of  the canonical  projection $\pi$  to $\X$ (still denoted  by $\pi$)   maps
$\X$ onto $\I$  bijectively. Moreover, $\I$ is  a  Borel set and   the map $\pi:\ \X\to\I$  and its inverse are Borel.
\item  For $x , y\in \X$ with  $x \not = y,$ $ L_x$ and $ L_y$ are disjoint.
\item  The 
 union of   $\mathcal L_x$  $x\in\X$ is  equal to  $(\D\setminus \{0\})^n.$


\end{enumerate}
\end{lemma}
\proof
To  each $x\in (\overline\D\setminus \{0\})^n,$ we associate
a finite  subset $\Sc(x)\subset (\overline\D\setminus \{0\})^n$ as follows:
$y\in\Sc(x)$  if and only if $y\sim x$ and  there are at least two  indexes $1\leq j <l\leq n$ such that $ |y_j|=1$ and $|y_l|=1.$

Observe that  if $x\sim x'$ then  $\Sc(x)=\Sc(x')$ and if $\Sc(x)\cap\Sc(x')\not=\varnothing$ then $x\sim x'.$
Moreover,  we  have the  following geometric interpretation:
$y\in\Sc(x)$ if and only if $y=\varphi_x(\zeta),$ where $\zeta$ is a vertex  of the  convex polygone  $\Pi_x.$  So 
the cardinality of  $\Sc(x)$ is equal to $m(x).$ In particular,
$0\leq \# \Sc(x)\leq n.$ 

On the other hand, the  weak hyperbolicity of the single  singularity  $\{0\}$ implies that for every $x\in (\overline\D\setminus \{0\})^n,$  there are  at least two lines among  $n$ lines $ \H_{x,j}$
which are  not parallel.  Suppose that  they are $\partial \H_{x,l}$ and $\partial \H_{x,l'}.$
So    the  convex polygone  $\Pi_x$ admits at least  one  (finite)  vertex.  Moreover, this  vertex is  either  the intersection of
$\partial \H_{x,l}$ and $\partial \H_{x,l'},$  or the intersection of one of them with another line $\partial \H_{x,j}.$
 Therefore,  $\Sc(x)\not=\varnothing,$ and hence $1\leq \# \Sc(x)\leq n.$

The  weak hyperbolicity of the single  singularity  $\{0\}$ also implies  that
for $y,z\in  (\overline\D\setminus \{0\})^n$ with $y\sim x,$ $z\sim x,$ if  $|y_j|=|z_j|$ for all $1\leq j\leq n,$ then $y=z.$
Indeed, write $y=\varphi_x(\zeta)$ and $z=\varphi_x(\xi)$ for some $\zeta,\xi\in\Pi_x.$
We infer  from   $|y_j|=|z_j|$ and  equality \eqref{e:dist-x_j-vs-y_j} that $\dist(\zeta,\partial \H_{x,j})=\dist (\xi,\partial \H_{x,j}).$ This  equality for $j=l$ and $j=l',$  coupled with the fact that two real lines
 $\partial \H_{x,l}$ and $\partial \H_{x,l'}$ are not parallel and  $\zeta,\xi\in\Pi_x,$ implies that $\zeta=\xi.$ So $y=z$ as asserted.

Therefore, 
we  order the  elements of  $\Sc(x)$ using a lexicographical order. More specifically, for $y,z\in\Sc(x)$    we say that $y\succ z$ if  and only if there is an index $1\leq j\leq n$ such that
$|y_k|=|z_k|$ for all $k<j$ and  $|y_j|>|z_j|.$  
Let $x^*$ be the greatest  element of  $\Sc(x)$ with respect to this total order.
Let $$\X:=\left\lbrace x^*:\quad x\in  (\overline\D\setminus \{0\})^n\right\rbrace.$$
We can check all properties of the lemma.
\endproof

The  following result is  one of the main ingredients  in the proofs of Theorem \ref{T:main} and Theorem  \ref{T:main_bis}.

\begin{lemma}\label{L:decomposition} Let $T$ be a  positive  harmonic  current $T$  directed  by $\Fc$ on $(\overline\D\setminus\{0\})^n$
such that  the mass of $T$ on $\D^n\setminus (r_0\D)^n$ is finite  for some $r_0\in (0,1).$
Then
there is a positive measure $\mu$ 
on $\X$ and positive
harmonic functions $\tilde h_x$ on $\mathcal L_x$ for $\mu$-almost every $x\in \X$  such that in $(\D\setminus\{0\})^n$
$$T =\int_\X T_xd\mu(x),\quad \text{where}\quad T_x := \tilde h_x [\mathcal L_x ].$$
Moreover, the mass of $T_x$ on  $\D^n\setminus (r_0\D)^n$ is $1$ for $\mu$-almost every $x\in \X.$
\end{lemma}

\proof 
Suppose without loss of generality that $T$ is defined on  a neighborhood $\U$ of $(\overline\D\setminus\{0\})^n$
as in Lemma \ref{L:transversal}.
By  Lemma \ref{L:quotient},  the map $\pi$ which
associates to a point in $L_x$  the image $[x]$ in $\I$ is a holomorphic map and
$T$ is directed by the fibers of $\pi.$ We  regard $\pi:\  (\D\setminus \{0\})^n\to \I$
as a  simple  foliation $\widehat\Fc$ whose  leaves are $\mathcal L_x\times \{x\} ,$  $[x]\in \I.$
Applying  Proposition \ref{P:decomposition} to $T$  in this context,
we obtain
the  following decomposition of $T.$ 
There is a positive measure $\mu'$ 
on $\I$ and positive
harmonic functions $\tilde h_x$ on $ L_x$ for $\mu'$-almost every $x\in \X$  such that in a neigborhood of  $\overline\D^n$
\begin{equation} \label{e:decomposition} T =\int_\I T_x d\mu'([x]),\quad \text{where}\quad T_x := \tilde h_x [\mathcal L_x ].
\end{equation}
Since  we know  by  Lemma \ref{L:transversal} that  the restriction of $\pi$ to the  Borel set $\X$ is Borel measurable  bijective map,
we  define  a  measure $\nu$ on $\X$  as   follows:
$$
\mu(A)=\mu'(\pi(A))\quad\text{for any  Borel set}\quad  A\subset\X.
$$
Consequently, the  above  decomposition of $T$  can be rewritten as
$$T =\int_\X T_x d\mu(x),\quad \text{where}\quad T_x := \tilde h_x [\mathcal L_x ].$$
 Finally, since    $T$ may be  regarded as a positive $\ddc$-closed  current on  an open neighborhood of  $(\overline\D^n\setminus (r_0\D)^n,$
 the mass of $T$ on the last set is  finite. Therefore, we multiply $\nu$ by a
suitable positive function $\theta (x)$ and divide $\tilde h_x$ by $\theta (x)$ in order to assume that
the mass of $T_x$ on $\D^n\setminus (r_0\D)^n$  is $1$ for $\mu$-almost every $x\in\X.$  This completes the proof of the lemma.

For the sake of  completeness, we give here an alternative  argument  which permit us  to get the decomposition    \eqref{e:decomposition} directly  without   using    Proposition \ref{P:decomposition}. 
Since $T$  is  a directed  current of bidegree $(1,1),$ we see that if $\alpha$ is any smooth form of
degree 1 or 2 on $\I,$ then $T \wedge \pi^* (\alpha) = 0.$

Consider the family $\mathcal F$ of all positive $\ddc$-closed $(1, 1)$-currents $R$ on
 $\U$ which are vertical in the sense that $R \wedge \pi^* (\alpha) = 0$ for any smooth form
$\alpha$ of degree 1 or 2 on $\I.$

\noindent {\bf Claim.} If $S$ is any current in $\mathcal F$ and $u$ is a smooth positive function on  $\I,$ then
$(u \circ \pi )S$ also belongs to $\mathcal F .$

Indeed, it is clear that $(u \circ \pi )S$ is positive and vertical. The only point to check is that $(u \circ \pi )S$ is $\ddc$-closed.

Define $\tilde u := u \circ \pi .$ Since $S$ is  $\ddc$-closed, we have $\ddc S = 0.$ Moreover,
as $S$ is vertical, we also get that $d\tilde u \wedge S = 0$ and $\dc \tilde u \wedge S = 0$ and
$\ddc\tilde u \wedge S = 0.$
Therefore, a straightforward calculation gives
$$\ddc (\tilde 
u S) = d(\dc \tilde u\wedge  S) - \dc (d\tilde 
u \wedge S) - \ddc \tilde u \wedge S + \tilde
u \ddc S = 0,$$
which proves the claim.

It follows from the claim that every extremal element in $\mathcal F$ is supported
by a fiber $\pi^{-1}([x])$  of $\pi,$ which is biholomorphic to the  convex  polygon  $\Pi_x\subset\C.$ A positive $\ddc$-closed current on a
Riemann surface is defined by a positive harmonic function.  We conclude that every
extremal element in $\mathcal F$ is  of the form $\tilde h_x [\mathcal L_x] ,$
where $\tilde h_x$ is a  positive  harmonic function on $\mathcal L_x.$
 The set of all positive $\ddc$-closed vertical currents $S\in\mathcal F$ such that  $\|S\|_{\D^n\setminus (r_0\D)^n}=1$  is  a convex compacts set.  
Therefore, by Choquet's representation theorem, $S$ is an average of those
extremal currents.
 The decomposition  \eqref{e:decomposition} follows.
\endproof

\begin{remark}\rm 
 We will see later on that    the mass  of  $T$ on $\D^n$ is  finite.  Hence,  we can even assume that  the mass of $T_x$ in $\D^n$ is $1$ for $\mu$-almost every $x\in \X.$
\end{remark}

For $\nu$-almost every $x\in\X,$  consider the function  $ h_x:\ \Pi_x\to\R^+$ given by
\begin{equation}\label{e:h_x} h_x(\zeta) :=\tilde h_x (\varphi_x(\zeta))\quad\text{for}\quad \zeta\in\Pi_x.
\end{equation}
So   $ h_x$ is a positive harmonic function  on $\Pi_x.$ For $x\in\overline\D^n,$  let $ P_x(\cdot,\cdot)$
be the Poisson kernel of $\Pi_x,$ that is,  $P_x:=P_{\Pi_x}.$
For $x\in \X$ and $r>0$  consider the following (eventually empty) sub-domain of $\Pi_x:$
\begin{equation}\label{e:Pi^r_x}  
\Pi^r_x:=\left\lbrace \zeta\in \Pi_x:\  \dist_x(\zeta)> r\right\rbrace.
\end{equation}
Note that $\Pi^r_x$ is  an (eventually empty) polygon  whose edges are parallel to some edges of  $\Pi_x.$

 Since $\D^n\Subset \U,$ we can find $r'_0>0$  such that  $\U':=((1+r'_0) \D)^n\Subset \U.$
For $x\in \X$ let $\Pi'_x:=\varphi^{-1}_x(\U').$

\begin{lemma}\label{L:r_1}
There exists $r_1>0$ such that the following  assertions hold for $x\in \X.$ 
\begin{enumerate}
 \item If $\Pi_x$  contains  a disc of radius $2r_1,$  then $\Pi^{r_1}_x\not=\varnothing.$ 
 \item $\Pi'_x$  contains a disc of radius $2r_1.$
\end{enumerate}
\end{lemma}
\proof
The first assertion holds  for all $r_1>0$    using \eqref{e:Pi^r_x}.

 Fix $x=(x_1,\ldots,x_n)\in \X.$
Let $ j=1,\ldots,k.$    If $x_j\not=0,$ following \eqref{e:H_x,j} let $\H'_{x,j}$ be the   open half-plane  defined by 
\begin{equation}\label{e:H'_x,j}\H'_{x,j}:=\{\zeta=u+iv\in\C:\ -a_ju+b_jv -\log|x_j|>\log(1+r'_0)\}=
 \{ \zeta\in\C:\ |\varphi_x(\zeta)_j|<1+r_0   \}.\end{equation}
 Consequently, $\partial\H_{x,j}$ and $\partial\H'_{x,j}$ are  parallel lines and  $\dist(\partial\H_{x,j},\partial\H'_{x,j})=|\lambda_j|^{-1}\log(1+r'_0).$
 Therefore, we get for $\zeta\in\Pi_x,$
 \begin{equation*}
 \dist (\zeta, \partial\H'_{x,j} ) = \dist (\zeta, \partial\H_{x,j} )+\log(1+r'_0).
\end{equation*}
 If  $x_j=0,$ then   set simply  $ \H'_{x,j}:=\C.$
As in  \eqref{e:phase-space} we have that
$ \Pi'_x=\bigcap_{j=1}^n\H'_{x,j}.$
 We  infer from the above consideration that
 for $\zeta\in\Pi_x,$
 \begin{equation*}
 \dist (\zeta, \partial\Pi'_{x} ) = \dist (\zeta, \partial\Pi_{x} )+\log(1+r'_0).
\end{equation*}
This  inequality  implies the second assertion for $0<r_1\ll ( \max_{1\leq j\leq n}|\lambda_j|^{-1})  \log(1+r'_0).$
\endproof
\begin{definition}\label{D:partition-X}\rm 
 Let $\X'$ be the set of all $x\in\X$ such that the convex polygon 
  $\Pi_x$ contains  a disc of radius $2r_1.$
\end{definition}
\begin{remark}\label{R:X'}\rm
 Roughly speaking,  the fact  that the convex polygon  $\Pi_x$ contains  a disc of  radius $2r_1$ means  that  $\Pi_x$ is non-degenerate (i.e., not so thin).
In  Section \ref{S:Main-estimates}, we  can obtain   good estimates  on  Poisson kernel  only for  such  polygons.  

If  $x\in  \X\setminus \X',$ then 
$\Pi_x$
does not contain  a disc of radius $2r_1.$ 
Consequently, the definition of $r_1$ in  Lemma \ref{L:r_1} implies that  
 $\P'_x$  contains a disc of radius $2r_1.$ This  means that by passing from $\Pi_x$ to $\Pi'_x$ if necessary, we may assume  that
 for ``every'' $x\in\X,$  
 $\Pi_x$ contains  a disc of  radius $2r_1.$  

\end{remark}
\begin{lemma} \label{L:Poisson-representation}
For $\mu$-almost every $x\in\X,$ the  following   properties hold.
\begin{enumerate}
 \item The   function $h_x $ is the Poisson integral of its boundary values, that is, 
$$
 h_x(\zeta)= \int_{\partial \Pi_x} P_x(\zeta,\xi)  h_x( \xi)d\Leb_1(\xi).
$$  
\item There is a constant $c>0$ independent of $x\in\X$ such that
$$
\int_{\xi\in\partial \Pi_x}   h_x(\xi)d\Leb_1(\xi) <c.
$$
\end{enumerate}
\end{lemma}
\begin{proof} {\bf Proof of assertion (1).}
We only consider  $x\in \X$  such that
the mass of $T_x$ on $\D^n\setminus (r_0\D)^n$ is 1.  By  Lemma  \ref{L:decomposition}, 
 $\mu$-almost every $x\in \X,$  
satisfies  this  condition. By definition, the mass of $T_x$ in $\D^n\setminus (r_0\D)^n$ is the mass
of the following positive measure in $\D^n\setminus (r_0\D)^n$:
$$T_x \wedge (idz_1 \wedge  d\bar z_1 +\ldots+ idz_n \wedge d\bar z_n ) = \tilde h_x(z) (idz_1 \wedge  d\bar z_1 +\ldots+ idz_n \wedge d\bar z_n )\wedge  [L_x].$$
Using the  parametrization \eqref{e:varphi_x} of $L_x$  by $\Pi_x,$ we get that the mass of this
measure is equal to the one of its pull-back to $\Pi_x.$ Using \eqref{e:h_x} and  writing  $\zeta=u+iv,$  the last measure  on $\Pi_x$ is
$$
h_x(\zeta)\big(e^{-2|\lambda_1|\dist_{a,1}(\zeta)} +\ldots+ e^{-2|\lambda_n|\dist_{a,n}(\zeta)}  \big)d\Leb_2(\zeta).
$$
Since the mass of $\D^n\setminus (r_0\D)^n$  is  the integral of the above expression on the domain $\Pi_x\setminus \Pi^{r_0}_x$ and $r_1<r_0,$
this mass is larger than the integral on  the sub-domain $\Pi_x\setminus \Pi^{r_1}_x.$  
Moreover,  $e^{-2|\lambda_1|\dist_{a,1}(\zeta)} +\ldots+ e^{-2|\lambda_n|\dist_{a,n}(\zeta)}\approx 1$ on $\Pi_x\setminus \Pi^{r_1}_x.$
Hence, there is   a constant $c>0$ independent of $x$ such that 
\begin{equation}\label{e:lower-bound-mass-T_x}
\int_{\Pi_x\setminus \Pi^{r_1}_x}   T_x \wedge (idz_1 \wedge  d\bar z_1 +\ldots+ idz_n \wedge d\bar z_n )\geq  c \int_{\Pi_x\setminus \Pi^{r_1}_x}h_x(\zeta)d\Leb_2(\zeta).
\end{equation}
 So  there is  a constant $c'>0$ independent of $x\in\X$ such that
 \begin{equation}\label{e:bounded-mass_r_1}
  \int_{\Pi_x\setminus \Pi^{r_1}_x}h_x(\zeta)d\Leb_2(\zeta)<c'.
 \end{equation}
 Since  $h_x$  is a  positive  harmonic  function  on $\Pi_x$  and continuous up to  the boundary $\partial\Pi_x,$ it follows from  
 Proposition \ref{P:Poisson} that one  of the following  two cases  happens.
 \begin{enumerate}
 \item  Case 1:  $\Pi_x$ is  bounded. In this case  we have  
\begin{equation*}
 h_x(\zeta)=\int_{\partial\Pi_x} P_x(\zeta,\xi) h_x(\xi)d\Leb_1(\xi)\qquad\text{for}\qquad  \zeta\in\Pi_x.
\end{equation*} 
 \item   Case 2: $\Pi_x$ is  unbounded.  Let $\phi_x$ be  a  biholomorphic map  from $\Pi_x$  onto $\D$ which sends $\infty$ to $1\in\partial \D.$
Then,   there is a constant $c_x\geq 0$ such that 
\begin{equation*}
  h_x(\zeta)=\int_{\partial\Pi_x} P_x(\zeta,\xi) h_x(\xi)d\Leb_1(\xi)  +c_x\Gamma_x(\zeta)\qquad\text{for}\qquad  \zeta\in\Pi_x.
\end{equation*} 
Here,  $  \Gamma_x$ is   the  positive harmonic  function   on $\Pi_x$   defined  in  \eqref{e:Gamma_x} by
 $$ \Gamma_x(\zeta):=  {1-|\phi_x(\zeta)|^2\over |\phi_x(\zeta)-1|^2}\qquad \text{for}\qquad \zeta\in\Pi_x.$$
\end{enumerate}
To  complete the proof of assertion  (1), we only need to check that if Case (2) happens, then $c_x  =0.$  There are  two subcases  to consider.

\noindent {\bf Subcase  $x\in \X':$} 
We infer from the above equality of Case (2) and  the fact that $h_x\geq 0,$   $P_x(\zeta,\xi)\geq 0$  that
\begin{equation*}
  h_x(\zeta)=\int_{\partial\Pi_x} P_x(\zeta,\xi) h_x(\xi)d\Leb_1(\xi)  +c_x\Gamma_x(\zeta)\geq c_x \Gamma_x(\zeta)\qquad\text{for}\qquad  \zeta\in\Pi_x.
\end{equation*} 
By   Proposition \ref{P:gamma},  $\Gamma_x(\zeta)\geq  c^\star_x$  for $\zeta\in \partial \Pi^{r_1}_x.$   Hence, 
we get that
$$
\int_{\Pi_x\setminus \Pi^{r_1}_x}h_x(\zeta)d\Leb_2(\zeta)\gtrsim c^\star_x c_x\int_{\partial \Pi^{r_1}_x} \Gamma_x(\zeta) d\Leb_1(\zeta)\gtrsim c^\star_x c_x \int_{\partial \Pi^{r_1}_x}d\Leb_1(\zeta).
$$
Since the first  integral is  finite and the last one is  infinite (as  $\Pi_x$ is  unbounded  and $\Pi^1_x\not=\varnothing$) and $c^\star_x>0,$ we infer that $c_x=0.$  This proves assertion (1).

Since $h_x$ is positive harmonic on an open neighborhood of $\overline\Pi_x,$  by Harnack's inequality  $h_x(\zeta)/h_x(\zeta-\xi)$
is  bounded  from  below  by a strictly  positive constant independent of $x$ for $|\xi|\lesssim 1.$
We  infer from   \eqref{e:bounded-mass_r_1}  that   its integral on $\partial \Pi_x$ is  also bounded by a constant.  Assertion (2) follows in this  subcase.

\noindent {\bf Subcase  $x\not\in \X':$}  
By Remark \ref{R:X'},  
$\Pi'_x$ is a unbounded polygon which 
contains   a disc of radius $2r_1.$  
Therefore, we  are still able  to apply  Proposition \ref{P:gamma} as in the previous  subcase in order to prove assertion (1).



Assertion (2) can be  proved in the same way as  in the previous subcase.

The proof of the lemma is  complete modulo    Proposition \ref{P:gamma}.
 \end{proof}  
 \begin{remark}\rm  \label{R:dim2}
 Lemma \ref{L:Poisson-representation} in  dimension $n=2$ has previously been obtained by Forn{\ae}ss--Sibony in  \cite[Proposition 1]{FornaessSibony10} (see also \cite[Lemma 4.2]{DinhSibony18}  for another proof).
 In their analysis, these authors  make  a full use of  the fact   that $\Pi_x$  is  independent of $x\in\X$ modulo  a translation.
 However,  this    perculiar fact in dimension $2$ does not hold in higher dimensions.
\end{remark}
For $0<r<1,$ 
let
\begin{equation}\label{e:F}
F(r):=  \int_{\B(0,r)}  T\wedge  \ddc \|x\|^2.
\end{equation} 
Consider also the function
\begin{equation}\label{e:f}
f(r):=
{1\over\pi r^2}F(r).
\end{equation}
By  \eqref{e:Lelong_bisbis} the Lelong number $\nu(T,0)$ of $T$ at $0$ is  $\limsup_{r\to 0} f(r).$
Let $$\K:=\{ (x,  \xi):\quad  x\in\X \quad \text{and}\quad \xi\in \partial \Pi_x \}.$$
For each   $s>0,$  consider   
the  function $K_s:\ \K\to\R^+$ given  by
\begin{equation}\label{e:K_s}
K_s(x,\xi):= 
\begin{cases} \int_{\zeta\in \Pi^s_x}e^{2s-2 \dist_x(\zeta)}P_x(\zeta,\xi)d\Leb_2(\zeta), &\ \ \text{if}\ \   \Pi^s_{x}\not=\varnothing;\\
 0, &\ \ \text{if}\ \   \Pi^s_x=\varnothing.
\end{cases}
\end{equation}

  \begin{lemma} \label{L:Lelong-representation}
 For  every $0<r<1,$
$$
f(r)\leq n \int_{x\in\X}\big(\int_{\xi\in\partial \Pi_x} K_{-\log r}(x,\xi)  h_x(\xi)d\Leb_1(\xi)\big) d\mu(x).
$$
\end{lemma}
\begin{proof}
Applying  Lemma  \ref{L:Poisson-representation} to \eqref{e:F} and using \eqref{e:h_x}, we see that    for $0<r\ll 1,$ 
 $$
F(r)= \int_{x\in\X} \int_{\zeta\in \Pi_x:\  |\varphi_x(\zeta)_j|\leq r,\ \forall 1\leq j\leq n} h_x(\zeta)  |\varphi'_x(\zeta)|^2  d\Leb_2(\zeta) d\mu(\alpha).
$$
On the other hand, we infer  from \eqref{e:varphi_x}  and \eqref{e:dist_x}--\eqref{e:dist_x-bis} that  for  $y=\varphi_x(\zeta),$  $|y|\leq r$ implies $ \dist_x(\zeta)\geq -\log r.$
Moreover,  using \eqref{e:h_x}) and \eqref{e:dist_x}--\eqref{e:dist_x-bis} again, we get that
 $$|\varphi'_x(\zeta)|^2=  e^{-2|\lambda_1|\dist_{x,1}(\zeta)} +\ldots+ e^{-2|\lambda_n|\dist_{x,n}(\zeta)}    \leq ne^{-2\dist_x(\zeta)} .
$$
Consequently,
 $$
 F(r)\leq n \int_{x\in\X} \int_{\zeta\in \Pi^{-\log{r}}_x} h_x(\zeta)  e^{-2\dist_x(\zeta)} d\Leb_2(\zeta) d\mu(\alpha).
 $$
 Applying    Lemma \ref{L:Poisson-representation}  to  the inner integral  of the last line
   and using \eqref{e:f}, the lemma follows.
 \end{proof} 
 \begin{remark}\rm
  \label{R:stronger-version-f}
  It is  worthy noting that the above proof also shows the following   estimate. For  every $0<r<1,$
$$
\hat f(r)\leq n \int_{x\in\X}\big(\int_{\xi\in\partial \Pi_x} K_{-\log r}(x,\xi)  h_x(\xi)d\Leb_1(\xi)\big) d\mu(x),
$$
where  $\hat  f(r):= {1\over\pi r^2}\widehat F(r)$ and $\widehat F(r):=      \int_{(r\D)^n}  T\wedge  \ddc \|x\|^2.$
This  estimate is  stronger than Lemma  \ref{L:Lelong-representation} because $\B(0,r)\subset  (r\D)^n.$
 \end{remark}

\section{Poisson kernels of convex polygons: main estimates}\label{S:Main-estimates}
Let  $\Gamma$ be  the  positive harmonic  function   on $\D$ given by
\begin{equation}\label{e:Gamma}
 \Gamma(\zeta):=  {1-|\zeta|^2\over |\zeta-1|^2}\qquad \text{for}\qquad \zeta\in\D.
\end{equation}
\begin{proposition}\label{P:Poisson}
 Let $\Omega\subset \C$ be   a (not necessarily bounded) convex polygon with the  Green function $G(\zeta,\xi).$ Following the model of $\Cc^2$-smooth bounded domains of  Proposition \ref{P:Poisson-classic}, let   the Poisson kernel  on $\Omega$ be the function
 \begin{equation}\label{e:Poisson-kernel}
  P(\zeta,\xi):= -\Nor_\xi G(\zeta,\xi)\qquad\text{for}\qquad  \zeta\in\Omega,\ \xi\in\partial \Omega.
 \end{equation}
 There are two cases.
 \begin{enumerate}
 \item {\bf Case  $\Omega$ is  bounded:}  Then,  for every positive function $u\in\Cc(\overline\Omega)$  which   is harmonic    on $\Omega,$ we have  
\begin{equation*}
 u(\zeta)=\int_{\partial\Omega} P(\zeta,\xi) u(\xi)d\Leb_1(\xi)\qquad\text{for}\qquad  \zeta\in\Omega.
\end{equation*} 
 \item  {\bf Case $\Omega$ is  unbounded:}  Let $\phi$ be  a  biholomorphic map  from $\Omega$  onto $\D$ which sends $\infty$ to $1\in\partial \D.$
Then,  for every   positive function $u\in\Cc(\overline\Omega)$ which is  harmonic  on $\Omega,$  there is a constant $c=c_u\geq 0$ such that 
\begin{equation*}
 u(\zeta)=\int_{\partial\Omega} P(\zeta,\xi) u(\xi)d\Leb_1(\xi)  +c(\Gamma\circ\phi)(\zeta)\qquad\text{for}\qquad  \zeta\in\Omega.
\end{equation*} 
 \end{enumerate}
Moreover, in both cases,  $P(\cdot,\xi)$ is a positive harmonic function  on $\Omega$ when $\xi\in\partial\Omega$ is  fixed.
\end{proposition}
\proof Let $\phi$ be  a  biholomorphic map  from $\Omega$  onto $\D.$
Since   $\Omega$ is a  convex polygon, $\phi$  extends  continuously to $\overline \Omega$ and $\phi|_{\partial\Omega}$ is  one-to-one  onto its  image in $\partial \D.$
Let $G_\D$ be the  Green function of the unit-disc $\D.$ It follows  from the definition of Green function that
 \begin{equation}\label{e:Green-function-invar}
  G(\zeta,\xi)=G_\D(\phi(\zeta),\phi(\xi))\qquad \text{for}\qquad (\zeta,\xi)\in\Omega\times\overline\Omega\setminus \Delta.
  \end{equation}
  
To  prove assertion (1), 
Let  $u$ be  a   function in $\Cc(\overline\Omega)$  which   is harmonic    on $\Omega,$ and let $\zeta\in\Omega.$
Observe that $\phi$ extends   to a  diffeomorphism   from  $\overline\Omega$  onto $\partial\D.$
Consider the  function $v:\ \overline \D\to \R$ defined  by 
$$v(\hat \zeta):= u(\phi^{-1}(\hat \zeta))\qquad \text{for}\qquad \hat \zeta\in\overline\D.$$
By Proposition  \ref{P:Poisson-classic} applied to $v,$ we have that
$$
v(\phi(\zeta))=\int_{\partial \D} P_\D(\phi(\zeta), \hat \xi)v(\hat \xi) d\Leb_1(\hat \xi)= \int_{\partial \D} -\Nor_{\hat \xi} G_\D(\phi(\zeta), \hat\xi)v(\hat \xi) d\Leb_1(\hat\xi).
$$
Using the  change of variable  $\hat \xi:=\phi(\xi)$  for $\xi\in \partial\Omega,$ 
we  see that
the  RHS of the last line is  equal to
$$
\int_{\partial \Omega} \big(-\Nor_{\hat \xi} G_\D(\phi(\zeta),\hat \xi) \big)_{\hat \xi= \phi(\xi)}v(\phi(\xi)) d\Leb_1(\phi(\xi)).
$$
Since $\phi$ is  conformal  on $\overline\Omega,$  we  see that   $d\Leb_1(\phi(\xi))|_{\partial \D}=d\Leb_1(\hat\xi)|_{\partial \Omega}|\phi'(\xi)|.$ 
Moreover,  using  identity \eqref{e:Green-function-invar} and
Proposition \ref{P:Poisson-classic}  we also get that 
$$ \big(-\Nor_{\hat \xi} G_\D(\phi(\zeta),\hat\xi) \big)_{\hat \xi= \phi(\xi)}=\big(-\Nor_{\xi} G_\Omega(\zeta,\xi) \big) |\phi'(\xi))^{-1}|=P_\Omega(\zeta,\xi) |\phi'(\xi))^{-1}|.$$
So  the last integral  is  equal to
$$
\int_{\partial \Omega }P_\Omega(\zeta,\xi) u(\xi) d\Leb_1(\xi).
$$
Consequently,    assertion (1) follows.

Now we turn to  the  proof of assertion (2).  Observe that $\phi$  maps $\partial\Omega$  bijectively  onto $\partial \D\setminus \{0\}.$
On the  other hand, for every   positive function $u\in\Cc(\overline\D\setminus \{1\})$ which is  harmonic  on $\D,$  there is a constant $c=c_u\geq 0$ such that 
$u$  is  the Poisson integral  of  the measure  $u(y)d\sigma_\D(y)+c \delta_1,$  where $\sigma_\D$ is  the Lebesgue  measure  on $\partial \D,$ and $\delta_1$ is  the  Dirac mass at $1.$  
So using \eqref{e:Gamma} and the explicit formula of $P_\D,$ we get 
\begin{equation*}
 u(\zeta)=\int_{\partial\Omega} P_\D(\zeta,\xi) u(\xi)d\Leb_1(\xi)  +c \Gamma(\zeta)\qquad\qquad\text{for}\qquad  \zeta\in\D.
\end{equation*} 
 Using this  and   identity \eqref{e:Green-function-invar} and arguing as in the proof of assertion (1),  the proof of   assertion  (2) is complete.
\endproof
We keep  the hypothesis and notations in   Theorem \ref{T:main}.

\begin{proposition} \label{P:gamma} 
Let $x\in\X$ be  such that $\Pi_x$ is  unbounded. 
Following  Proposition \ref{P:Poisson}, let $\phi_x$ be  a  biholomorphic map  from $\Pi_x$  onto $\D$ which sends $\infty$ to $1\in\partial \D.$
 Let  $\Gamma_x$ be   positive harmonic  function   on $\Pi_x$   defined  by
 \begin{equation}\label{e:Gamma_x}  \Gamma_x:= \Gamma\circ \phi_x,
  \end{equation}
where $\Gamma$ is  given in \eqref{e:Gamma}.
Then,   there is a constant $c^\star_x> 0$ dependent on $x$  such that 
\begin{equation*}
  \Gamma_x(\zeta)\geq c^\star_x \qquad\text{for}\qquad  \zeta\in\Pi^{r_1}_x  .
\end{equation*} 
\end{proposition}
\proof We rephrase the problem  differently but equivalently. So
   we only need to prove  that for every $x$ as in the assumption, there is a constant  $c^\star_x>0$ dependent on $x\in\X$  such that
 there is   a biholomorphic map  $\tilde\phi_x$ from the upper-half plane  $\H:=\{\zeta\in\C:\ \Im\zeta>0\}$ onto  $\Pi_x$ sending $\infty$ to $\infty$   such that
 \begin{equation}\label{e:empty}
 \tilde\phi_x(\H_{c^\star_x})\cap \Pi^{r_1}_x=\varnothing,\qquad\text{where}\qquad \H_c:=\{\zeta\in\C:\ \Im \zeta\in(0,c)\}.
 \end{equation}
Indeed,  since $\hat \phi_x:=\tilde \phi_x\circ\phi_x$ is  a biholomorphic  map from  $\H$ onto $\D$ sending $\infty$  to $1\in\partial \D,$  
we infer from   \eqref{e:Gamma} that
$$
\Gamma(\hat\phi_x(\zeta))=\Im   \zeta\qquad\text{for}\qquad\zeta\in \H.
$$
Therefore,    it follows that
$$
\Gamma_x(\zeta)=\Im  \tilde\phi_x^{-1}(\zeta)\qquad\text{for}\qquad\zeta\in \Pi_x.
$$
Thus  \eqref{e:empty} implies the proposition.

Next,  
we will prove that  there is constant $c^\star_x>0$ dependent on $x\in \X$  such that for $\theta\in\C$ with $|\theta|<c^\star_x,$
\begin{equation}\label{e:empty_bis}
|\tilde \phi_x(\zeta)-\tilde \phi_x(\zeta+\theta)|\leq 1 \qquad\text{for}\qquad \zeta\in\partial\H.
\end{equation}
Taking  \eqref{e:empty_bis} for granted,   \eqref{e:empty}  will follows because  $\tilde \phi_x(\zeta)\in\partial\Pi_x$ for $   \zeta\in\partial\H ,$ which completes  the
proof of the proposition.

To prove    the reduction \eqref{e:empty_bis}. Let $w_1,\ldots,w_k$ be all finite vertices of the convex polygon $\Pi_x$ in counterclockwise order and
set  $w_0=w_{k+1}:=\infty,$  see  Figure  \ref{fig:M4}. 
Let $\alpha_j:=\measuredangle \big(\overrightarrow {w_jw_{j+1}}, {\overrightarrow {w_jw_{j-1}}}\big),$ for  $1\leq j\leq k+1,$  be their  corresponding interior angles in counterclockwise order, with the convention that $w_{k+2}:=w_1.$
Observe  that $\alpha_j\in  (0,\pi)$ for $1\leq j\leq k$   and $\alpha_{k+1}\in (-\pi,0].$
Write  
\begin{equation}\label{e:gamma_j}
\alpha_j:= {\pi\over\gamma_j}\qquad\text{for}\qquad 1\leq j\leq k. 
\end{equation}
So $\gamma_1,\ldots,\gamma_k>1.$
By the classical Schwarz-Christoffel  formula (see e.g. \cite[formula (22) p.10]{DT}), we can write
\begin{equation}\label{e:SC}
\tilde \phi_x(\zeta)=c'_x+c''_x\int^\zeta \prod\limits_{j=1}^k (\eta-z_j)^{{1\over \gamma_j}-1}d\eta,
\end{equation}
for some complex  constants $c'_x$ and $c''_x,$ where $z_1,\ldots,z_k\in\partial \H$ and $\tilde\phi_x(z_j)=w_j$ for $1\leq j\leq k.$

\begin{figure}
\centering
\begin{minipage}{.3\textwidth}
\begin{tikzpicture}
 \draw [->] (-4,0) -- (3.5,0);
 \draw (-4,0)  node [below] {$-\infty$};
  \draw (3.5,0)  node [below] {$\infty$};
 \draw (-3,0)  node [below] {$z_1$};
 \draw (-3,0) node {$\bullet$};
 \draw (-2,0) node [below] {$z_2$};
  \draw (-2,0) node {$\bullet$};
 \draw  (0,0)  node [above] {$\ldots$};
 \draw (1.9,0)  node [below] {$z_{k-1}$};
 \draw (1.9,0) node {$\bullet$};
 \draw (2.5,0) node [below] {$z_{k}$};
 \draw (2.5,0) node {$\bullet$};
 \draw (0,1)  node  [above]{\Large $\H$};
\end{tikzpicture}
 \end{minipage}
\hspace{5cm}
\begin{minipage}{.2\textwidth}
\begin{tikzpicture}
        \coordinate[label=below:$w_{k-1}$] (A) at (1,0);
        \coordinate[label=below:$w_k$] (B) at (3,0);
        \coordinate (C) at (0,1.7);
        \coordinate (D) at (0.5,0.5);
        \coordinate (w_0) at (3,4.5);
        \coordinate[label=above:$w_1$] (w_1) at (2,4);
        \coordinate[label=left:$w_2$] (w_2) at (1,3.5);
         \coordinate (w_k+1) at (4,0);
        \draw (w_1) node {$\bullet$};
 \draw (w_2) node {$\bullet$};
  \draw (C) node {$\bullet$};
   \draw (D) node {$\bullet$};
    \draw (A) node {$\bullet$};
     \draw (B) node {$\bullet$};
\draw (w_0)--(w_1)  node[
    currarrow,
    pos=0.5,sloped,xscale=-1] {};
\draw (w_2)--(w_1);
\draw (C)--(w_2);
\draw[dashed] (C)--(D);
\draw (D)--(A);
\draw (A)--(B);
\draw (B)--(w_k+1) node[
    currarrow,
    pos=0.5,sloped] {};




\draw (2,2)  node  [above]{\Large $\Pi_x$};
    \end{tikzpicture}
\end{minipage}

\caption{On the left: the upper-half plane $\H$ and the points $z_1,\ldots, z_k\in\partial \H ,$ where
the  dotted points correspond to intermediate points $z_3,\ldots,z_{k-2}.$
On the right: the image of $\H$ by  $\tilde \phi_x:$ the  unbounded convex polygon $\Pi_x$ with vertices  $w_1=\tilde\phi_x(z_1),\ldots,w_k=\tilde\phi_x(z_k),$   where the  dashed line  corresponds  to
the (not necessarily aligned) intermediate points   $w_3,\ldots,w_{k-2}.$} \label{fig:M4}
\end{figure}
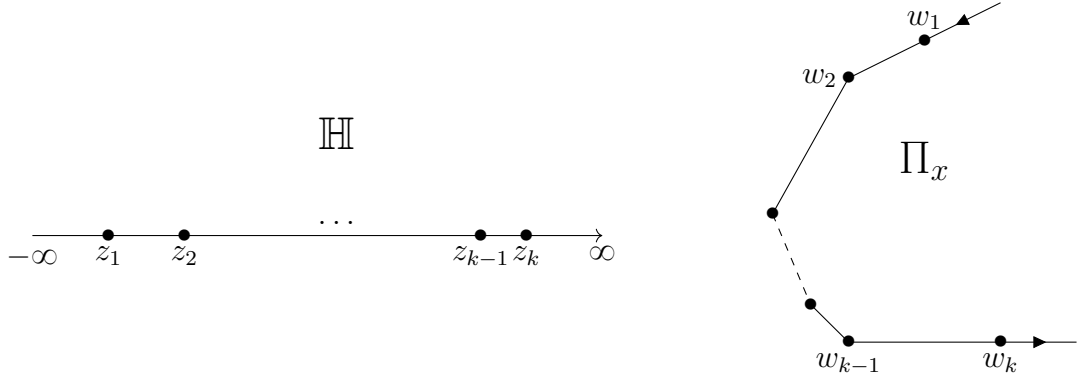

Clearly, $c''_x\not=0.$
Consequently,  we   infer from  \eqref{e:SC} that for every $\zeta\in\partial\H$ and $\theta \in  \C,$ 
\begin{equation}\label{e:SC-bis}
|\tilde \phi_x(\zeta)-\tilde \phi_x(\zeta+\theta)|=|c''_x|\big|\int^{\zeta+\theta}\limits_\zeta \prod\limits_{j=1}^k (\eta-z_j)^{{1\over \gamma_j}-1}d\eta\big|.
\end{equation} 
In order  to  prove  \eqref{e:empty_bis}, we need the following auxiliary result.
 
\begin{lemma}
 \label{L:Real-analysis} Let     $p\geq 1$  and let $s_1,\ldots,s_p \in(-1,0).$
 Let $t_1<\ldots <t_p$ be real numbers.
 Then, for  every $\epsilon>0,$ there is $\delta>0$ such  that for  $a\in[t_l-1,t_p+1],$ we have 
 $$\int_a^{a+\delta} \prod\limits_{j=1}^p  |t-t_j|^{s_j} dt<\epsilon.$$
\end{lemma}
\proof
The  proof is  elementary and we leave it to the interested  reader.
\endproof
Resuming the  proof of \eqref{e:empty_bis}, we consider two cases  according to  the position  of $\zeta\in\partial \H$
with respect  to   the set $Z:=\{z_1,\ldots,z_k\}\subset \partial\H.$
Observe  that $z_1<\cdots <z_k$  since $\tilde\phi_x(z_j)=w_j$ and the $w_j$ are in counterclockwise order. By a change of variables,
we rewrite  \eqref{e:SC-bis} as
\begin{equation}\label{e:SC-bisbis}
|\tilde \phi_x(\zeta)-\tilde \phi_x(\zeta+\theta)|=\int\limits^{|\theta|}_{t=0}\prod\limits_{j=1}^k \big| (\zeta+{\theta\over |\theta|} t)-z_j\big|^{{1\over \gamma_j}-1}dt.
\end{equation}

\noindent {\bf  Case $\dist(\zeta, Z)\ll 1.$ }

Let $1\leq l<m \leq  k$ be such that  $\dist(\zeta,z_j)\ll 1$ for $l\leq j\leq m$
and  $\dist(\zeta,z_j)\geq 1$ otherwise.    So  $ \big| (\zeta+{\theta\over |\theta|} t)-z_j\big|\gtrsim 1$  for $j\not\in [l,m],$ 
we deduce from \eqref{e:SC-bisbis} that
\begin{equation*}
|\tilde \phi_x(\zeta)-\tilde \phi_x(\zeta+\theta)|\lesssim |c''_x|\int\limits^{|\theta|}_{t=0}\prod\limits_{j=l}^m \big| (\zeta+{\theta\over |\theta|} t)-z_j\big|^{{1\over \gamma_j}-1}dt.
\end{equation*}
Moreover, the RHS is  dominated  by a constant times $$|c''_x|\int_{\zeta-|\theta|}^{\zeta+|\theta|}   |s-\zeta|^{\sum_{j=l}^m ({1\over \gamma_j}-1)}ds.$$
Using \eqref{e:gamma_j} we may apply Lemma \ref{L:Real-analysis}. Consequently, the last integral  
 is  small  provided that $c^\star_x:=c$ is  small enough  for $\theta\in\C$ with $|\theta|\in (0,c).$ This proves \eqref{e:empty_bis} in this case.
 
\noindent {\bf  Case $\dist(\zeta, Z)\geq 1.$ }
So  for a constant $0<c\ll 1,$ we  get  that 
$\big| (\zeta+{\theta\over |\theta|} t)-z_j\big|\geq |\zeta-z_j|-|t|\geq   1-c>0$  for $t\in[0,|\theta|]$ and $|\theta|<c.$
Therefore, we deduce from \eqref{e:SC-bisbis} that
\begin{equation*}
|\tilde \phi_x(\zeta)-\tilde \phi_x(\zeta+\theta)|\leq c'c (1-c)^{\sum_{j=1}^k ({1\over \gamma_j}-1)},
\end{equation*}
where $c'$ is a constant depending only on $n.$
Choosing $0<c\ll1$ and $c^\star_x:=c$ small enough,  \eqref{e:empty_bis} holds in this last case. 
\endproof

The  following result due  to  Widder  \cite{Widder}  gives  the Poisson kernel for strips.   
\begin{proposition}\label{P:Widder}
 For $(a,b,c,d)\in\R^4$ with $a^2+b^2>0$ and  $c<d,$ consider  the  strip 
 $$\S=\S_{a,b,c,d}:=\left\lbrace \zeta=u+iv\in\C:\   c<au+bv <d \right\rbrace,$$
 which is  limited by two   parallel  lines  $L_1=\{au+bv=c\}$ and  $L_2=\{au+bv=d\}.$
 Let $R:=\dist(L_1,L_2)$ be the distance between $L_1$ and $L_2.$ See Figure
 \ref{fig:M5}. 
 Then the following  assertions hold:
 \begin{enumerate}
  \item The Poisson  kernel  of $\S$ is  given  by
  $$
  P_{\S}(\zeta,\xi)={\pi\over R}\cdot\,{\sin{\big( {\pi \dist(\zeta, \zeta_\xi) \over R}\big)} \over \cosh{\big({\pi \dist(\xi,\zeta_\xi)\over R}\big)}- \cos{\big( {\pi \dist(\zeta, \zeta_\xi) \over R} \big)}}
  \quad\text{for}\quad \zeta\in\S,\quad \xi\in \partial \S. 
  $$
  Here, if   $\xi\in L_j$  then  $\zeta_\xi$ is the  orthogonal projection of $\zeta$ onto $L_j.$ 
  \item  For $\zeta\in\S$ and $ \xi\in \partial \S,$  
  we have
  $$
  P_{\S}(\zeta,\xi)\leq  2\cdot\,{\dist(\zeta,\partial \S)\over ( \dist(\zeta,\xi))^2}.
  $$
 \end{enumerate}

\end{proposition}

\begin{figure}
\centering
\usetikzlibrary {angles,calc,quotes}
\begin{tikzpicture}[angle radius=.75cm]
\node (A) at  (-2,0){};
\node (B)  at ( 3,.5)[right] {line $L_1$};
\node (C) at (-2,3) {} ;
\node (D) at ( 3,3.5) [right] {line $L_2$};
\node (Z) at (0.8,2.3) [above]  {$\zeta$};
\node (H)  at (1, 0.3) [below] {$\zeta_\xi$};
\draw  (Z)-- (H);
\node (P) at (1.2,0.32) {};
\node (R) at  (0.99, 0.4) {};
\node  (Q)  at (1.09, 0.41) {};
\draw (1.2,0.32)-- (1.18,.52);
\draw  (1.18,.52) --(0.98, 0.5) ;
\node (Z) at (0.8,2.3)   {$\bullet$};
\node (H)  at (1, 0.3) {$\bullet$};
 \node (X) at(-0.5,0.15) [below] {$\xi$};
 \node (X) at(-0.5,0.15)  {$\bullet$};
\draw  (A)--(B);
\draw  (C)--(D);
\end{tikzpicture}

\caption{We are given a   strip 
 $\S=\S_{a,b,c,d} $
   limited by two   parallel  lines  $L_1=\{au+bv=c\}$ and  $L_2=\{au+bv=d\},$ a point  $\zeta\in\S$ and
   a point $\xi\in \partial \S.$ In this  figure $\xi\in L_1,$ and hence $\zeta_\xi$ is  the orthogonal projection of $\zeta$ onto $L_1.$ Moreover, in this  figure we see that $\dist(\zeta,\partial\S)=\dist(\zeta,L_2).$} \label{fig:M5}
\end{figure}
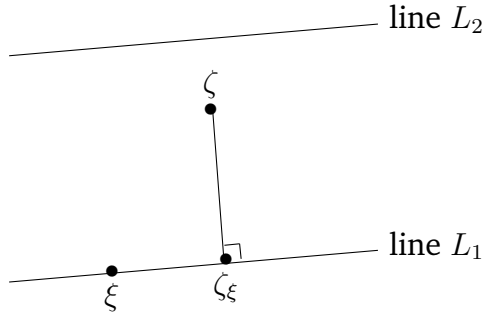
\proof Using  a rotation and  a translation, we may suppose that
the strip  $\S$ is given by   $\S_R:=\left\lbrace \zeta=u+iv\in\C:\   0<v<R \right\rbrace.$
  The change of variable $\zeta\mapsto {\pi \zeta\over  R}$ maps $\S_R$ biholomorphically onto $\S_\pi.$ 
Using this and   the explicit  formula   of the  Poisson kernel of $\S_\pi$ established  in \cite[formula (1)]{Widder}, 
we  get that
$$
P_{\S_R}(\zeta,\xi)={\pi\over R}\cdot\,{\sin{\big( {\pi |\Im (\zeta-\xi)| \over R}\big)} \over \cosh{\big({\pi \Re(\zeta-\xi)\over R}\big)}- \cos{\big( {\pi |\Im(\zeta-\xi)| \over R} \big)}}\quad\text{for}\quad \zeta\in\S_R,
\quad \xi\in \partial \S_R.
$$
Since $ |\Im (\zeta-\xi)|= \dist(\zeta, \zeta_\xi)$ and   $ |\Re (\zeta-\xi)|= \dist(\xi, \zeta_\xi),$
assertion (1) follows from this  formula.

Since  $ {2\over \pi}\leq \min (t,\pi-t)\sin{t}\leq  \min (t,\pi-t)$ for $t\in[0,\pi],$ we infer that
$$
 {2\over \pi}\leq{\sin{\big( {\pi \dist(\zeta, \zeta_\xi) \over R}\big)}\over  {\pi \dist(\zeta, \zeta_\xi) \over R}}\leq 1.
$$
Next, observe that  $\cosh t\geq 1\geq \cos t$ for $t\in\R.$ 
Moreover,   Taylor expansion of the  function $\cosh t$ 
gives that  $\cosh t\geq  t^2$ for $t\in\R.$
Writing  $1-\cos t=2\sin^2{t\over 2},$ we see that
${4\over \pi^2} \leq {2( 1-\cos t)\over t^2}\leq 1$ for $t\in [0,\pi].$
Using  the above  estimates, we obtain, for $\zeta\in\S$ and $ \xi\in \partial \S$ that
\begin{eqnarray*}
\cosh{\big({\pi \dist(\xi,\zeta_\xi)\over R}\big)}- \cos{\big( {\pi \dist(\zeta, \zeta_\xi) \over R} \big)}&=&
\big(\cosh{\big({\pi \dist(\xi,\zeta_\xi)\over R}\big)}-1\big)
+\big( 1- \cos{\big( {\pi \dist(\zeta, \zeta_\xi) \over R} \big)}\big)\\
&\geq &\big({\pi \dist(\xi,\zeta_\xi)\over R}\big)^2 +{1\over 2}\big( {\pi \dist(\zeta, \zeta_\xi) \over R} \big)^2\\
&= & {1\over 2}\big({\pi \dist(\xi,\zeta_\xi)\over R}\big)^2+{1\over 2}\big( {\pi\dist(\xi,\zeta)\over R}\big)^2\geq {1\over 2}\big( {\pi\dist(\xi,\zeta)\over R}\big)^2,
 \end{eqnarray*}
where the  equality in the last line holds by  Pythagorean Theorem. Inserting  these estimates in the  expression of $P_{\S}(\zeta,\xi)$
given in assertion (1), 
we obtain  assertion (2).    
\endproof

The  following   results  presents the  basic technique   to compare  Green functions and Poisson kernels  (see e.g. Krantz \cite{Krantz} for the princple in the case of smooth  bounded domains in $\R^N$). 

\begin{proposition}\label{P:estimate(1)}
  Let $\Omega_1$ and $\Omega_2$ be  (not necessarily  bounded)  convex polygons in $\C.$
  Suppose that $\Omega_1\subset \Omega_2$ and  there is a point $\xi\in\partial\Omega_1\cap  \partial\Omega_2 $ which is not  a vertex of $\Omega_1$ and $\Omega_2.$
  Then, for every $\zeta\in\Omega_1,$ 
 \begin{equation*}
    P_{\Omega_1}(\zeta,\xi) \leq  P_{\Omega_2}(\zeta,\xi) .
 \end{equation*}

\end{proposition}
\proof
By Definition \ref{D:Green}  and  by Proposition \ref{P:Poisson}  (see also identity \eqref{e:Green-function-invar}), 
we see that   $G_{\Omega_1} (\zeta, \cdot)$  equals $0$ on $\partial \Omega_1.$ Moreover,    by the Maximum  Principle for harmonic functions, $G_{\Omega_2} (\zeta, \cdot)\geq 0$
on $\Omega_2.$
Since     $\Omega_1\subset \Omega_2,$ it follows that   $G_{\Omega_2} (\zeta, \cdot)\geq 0$ on $\partial \Omega_1.$
Since the  function $G_{\Omega_1} (\zeta, \cdot)- G_{\Omega_2} (\zeta,\cdot) $  is  harmonic on
$\Omega_1,$
it follows from the Maximum  Principle again  that $G_{\Omega_2} (\zeta,\cdot)- G_{\Omega_1} (\zeta,\cdot) \geq 0$ on $\Omega_1.$ 
On the  other hand, since $\xi\in\partial\Omega_1\cap  \partial\Omega_2, $  we infer from  Definition \ref{D:Green} again that  $ G_{\Omega_2} (\zeta,\xi)- G_{\Omega_1} (\zeta,\xi) =0-0=0.$ Therefore, by the Hopf lemma (see for example \cite[p. 28]{ABR}) applied to the function 
$G_{\Omega_2} (\zeta,\cdot)- G_{\Omega_1} (\zeta,\cdot),$ we get that 
 $$-\Nor_\xi G_{\Omega_2} (\cdot, \xi) - \Big(-\Nor_\xi G_{\Omega_1} (\cdot, \xi)\Big) \geq 0.$$
This, combined with  \eqref{e:Poisson-kernel}, completes the proof.
\endproof

The  next result describes the complete behavior of the Poisson kernel of  a phase space in dimension $n=2.$

\begin{proposition}\label{P:estimate(2)}
  Suppose that $n=2.$  So  $\Pi_x$ is  a  sector in $\C$  with  aperture angle  $\pi/\gamma$ for some $\gamma >1.$  Then there is a constant $c>1$ which depends only on $\lambda_1$ and $\lambda_2$ 
 such that for every $x\in \X,$ $\zeta\in \Pi_x$ and $\xi\in\partial\Pi_x,$
 \begin{equation*}
 c^{-1}\leq   P_x(\zeta,\xi): {\dist_x(\zeta)\over |\zeta-\xi|^2}\,\cdot\Big({ \min\big( \dist^\star_x(\zeta),\dist^\star_x(\xi)\big)\over \max\big( \dist^\star_x(\zeta), \dist^\star_x(\xi)\big)} \Big)^{\gamma-1} 
 \leq c  .
 \end{equation*}

\end{proposition}
\proof
We  use  the proof and  the notation of Lemma 3.3 in \cite{NguyenVietAnh18a}. By using  a translation and  a rotation in $\C$ if necessary,  we may assume  without loss of generality that
$\lambda_1=1$ and $\lambda_2=a-ib$ with $a,b\in\R$ and $b>0.$ So  the aperture angle of $\Pi_x$ is $\alpha:=\pi/\gamma=\arctan{(-b/a)}  $ and  $\Pi_x$ is  given by 
$$\Pi_x:=\left \{ \zeta\in\C:   \arg \zeta \in(0,\pi/\gamma)  \right\} . $$
Write $\zeta:=u+iv$ with $u,v\in\C.$ So 
 $\Pi_x=\{\zeta\in\C:\  v>0 \quad\text{and}\quad bu+av>0\}$
 and $\partial \H_{x,1}=\{v=0\},$  $\partial  \H_{x,2}=\{bu+av=0\}.$ 
The
map 
\begin{equation}\label{e:u,v_vs_U,V}
\tau: z \mapsto z^\gamma
\end{equation}
 maps the sector $\Pi_x$  to the upper half plane $\H:=\{U+iV\in\C:\  V>0\}$ with coordinates $(U, V ).$ Write 
 $$y:=\tau(\xi)\qquad\text{
 and}\qquad  U+iV:=\tau(\zeta).$$
 So $y$ lies on the real line  $\partial \H=\{U+iV\in\C:\  V=0\}.$   
  Note that
  \begin{equation}\label{e:dists-in-2D}
  \dist_x(\zeta)\approx \min\{v, bu+av\}\qquad\text{and}\qquad  \dist^\star_x(\zeta)\approx \max\{v, bu+av\}.
  \end{equation}
  We  deduce from    \eqref{e:u,v_vs_U,V} and  $y:=\tau(\xi)$  that 
  $$
  P_\H(U+iV,y)dy=P_x(\zeta,\xi)d\Leb_1(\xi)
  $$
  and  $dy=|\xi|^{\gamma-1}d\xi.$  So  it  follows  from  the  explicit  formula  $P_\H(U+iV,y)= {V\over V^2+( y-U)^2} $ that  
  \begin{equation}\label{e:change-Poisson}
   P_x(\zeta,\xi)={V\over V^2+( y-U)^2}|\xi|^{\gamma-1}.
  \end{equation}
Let   $c_1, c_2,c_3 >1$ be constants large enough  with $c_3>c_2$   given by   Lemma 3.3 in \cite{NguyenVietAnh18a}.   
By  assertion (1) of  that lemma  and by \eqref{e:dists-in-2D}, we get that
\begin{equation}  {1\over c_1}\leq {(\dist^\star_x(\zeta))^\gamma\over \sqrt{V^2+U^2}}\leq c_1\qquad \text{and}\qquad   {1\over c_1}\leq {(\dist^\star_x(\zeta))^{\gamma-1}\dist_x(\zeta) \over V}\leq c_1 . 
\end{equation}
Note that by equality   $y=\tau(\xi)=\xi^\gamma$ and by  \eqref{e:dists-in-2D},
\begin{equation}\label{e:conversion}  (1+|y|)^{1/\gamma}\approx  |\xi|\qquad \text{and}\qquad\dist_x^\star(\xi)\approx |\xi|.
\end{equation}
According  to  Lemma 3.3 in \cite{NguyenVietAnh18a}, we consider  four  cases.
\\
\noindent{\bf  Case} $ \dist^\star_x(\zeta)\geq  c_2|\xi|$ :
In this case        $\dist^\star_x(\zeta)\approx  |\zeta|$   and  $|\zeta-\xi|\approx \dist^\star_x(\zeta).$   This, combined with 
 assertion (2) of  Lemma 3.3 in \cite{NguyenVietAnh18a} and  \eqref{e:change-Poisson} and \eqref{e:dists-in-2D},  \eqref{e:conversion}, implies that
$$
P_x(\zeta,\xi)\approx {\dist_x(\zeta)\over( \dist^\star_x(\zeta))^{\gamma+1}}|\xi|^{\gamma-1}\approx  {\dist_x(\zeta)\over |\zeta-\xi|^2}\,\cdot\Big({|\xi|\over \dist^\star_x(\zeta)}\Big)^{\gamma-1}.
$$
So  the conclusion of the lemma is  true in this case.
\\
\noindent{\bf  Case} $ \dist^\star_x(\zeta)\leq  c^{-1}_2|\xi|$ : In this case         $|\zeta-\xi|\approx |\xi| \approx \dist^\star_x(\xi).$   This, combined with 
 assertion (3) of  Lemma 3.3 in \cite{NguyenVietAnh18a} and  \eqref{e:change-Poisson}  and  \eqref{e:dists-in-2D},  \eqref{e:conversion}, implies that
$$
P_x(\zeta,\xi)\approx {\dist_x(\zeta)( \dist^\star_x(\zeta))^{\gamma-1}|\xi|^{\gamma-1}\over |\xi|^{2\gamma}}   \approx
{\dist_x(\zeta)\over |\zeta-\xi|^2}\,\cdot\Big({ \dist^\star_x(\zeta)\over  |\xi|}\Big)^{\gamma-1}.
$$
So  the conclusion of the lemma is  also true in this second  case.

\noindent{\bf  Case} $  c^{-1}_2|\xi|\leq  \dist_x(\zeta), \dist^\star_x(\zeta)\leq  c_2|\xi|$ : In this case    $\dist_x(\zeta)  \approx \dist^\star_x(\xi)\approx |\xi|$ and
$|\zeta-\xi|\approx |\xi|. $    This, combined with 
 assertion (4) of  Lemma 3.3 in \cite{NguyenVietAnh18a} and  \eqref{e:change-Poisson}  and  \eqref{e:dists-in-2D},  \eqref{e:conversion}, implies that
$$
P_x(\zeta,\xi)\approx {|\xi|^{\gamma-1}\over |\xi|^\gamma}   \approx 
{\dist_x(\zeta)\over |\zeta-\xi|^2}.
$$
So  the conclusion of the lemma is  also true in this third case.

\noindent{\bf  Case}  $     \dist_x(\zeta)\leq  c^{-1}_3|\xi|$ and  $  c^{-1}_2|\xi|\leq   \dist^\star_x(\zeta)\leq  c_2|\xi|$ :
 Following the proof of  assertion (5) in   Lemma 3.3 in \cite{NguyenVietAnh18a},  we may suppose  without loss of generality that $v\leq  bu+av.$ For  every $1\leq v\leq c_3^{-1}(1+|y|)^{1/\gamma},$
there exists  a    solution  $\hat u:=\hat u(y,v)$ of the following   equation
$$  \widehat U=y,\qquad\text{where}\  \widehat U+i\widehat V=(\hat u+iv)^\gamma=\tau(\hat u+iv)$$
which satisfies
   $$ c_2^{-1}(1+|y|)^{1/\gamma}\leq \hat u(y,v),\rho(y,v)\leq c_2(1+|y|)^{1/\gamma},$$
  where $\rho(y,v):=b\hat u(y,v)+av.$ 
    Let $\rho:=\rho(y,v) .$ 
    
   There are two subcases.
    
    \noindent {\bf Subcase:} $y\geq 0.$
    
    As $y=\tau(\xi),$  we see that  $\xi$ lies on the   ray  $\{v=0,\ u>0\}.$ 
    So by \eqref{e:dists-in-2D},
    $|\xi|\approx |y|^{1/\gamma}\approx \dist^\star_x(\zeta)$  and $\dist_x(\zeta)\approx  v.$ 
    By Lemma 3.4  in \cite{NguyenVietAnh18a} applied to $y+i\widehat V= \widehat U+i\widehat V=(\hat u+iv)^\gamma$
    and $y=\xi^\gamma,$ we have that
    $$
    \widehat V\approx  |(\hat u(y,v)+iv)-\xi|\xi^{\gamma-1}.
    $$
    On the other hand,  inequality (13)  in  Lemma 3.4  and assertion (1)  of  Lemma 3.3  in \cite{NguyenVietAnh18a} together imply that
    $$
    \widehat V\approx  V \approx  v\xi^{\gamma-1}.
    $$
    Hence, we infer that
    \begin{equation}\label{e:equiv-v}|(\hat u(y,v)+iv)-\xi|\approx   |v|
    \end{equation}   when $c_3$ is large enough. Consequently, we get that
    \begin{equation*}
     |\zeta-\xi|=|(u+iv)-\xi|\leq  |(\hat u(y,v)+iv)-\xi| +|u-\hat u(y,v)|\lesssim |v|+|(bu+av)-\rho|.
    \end{equation*}
On the other hand, since  $|\zeta-\xi| \geq \dist(\zeta, \partial \Pi_x)\gtrsim  \dist_x(\zeta)\approx  v,$ we deduce from \eqref{e:equiv-v} that
$$
  |\zeta-\xi|\approx    |\zeta-\xi| +v\approx |\zeta-\xi|+ |(\hat u(y,v)+iv)-\xi|\geq  |u-\hat u(y,v)|\gtrsim |v|+|(bu+av)-\rho|.
$$
So $ |v|+|(bu+av)-\rho|\approx |\zeta-\xi|.$
    This, combined with 
 assertion (5) of  Lemma 3.3 in \cite{NguyenVietAnh18a} and  \eqref{e:change-Poisson}  and  \eqref{e:dists-in-2D},  \eqref{e:conversion}, implies that
$$
P_x(\zeta,\xi)\approx {v\over v^2 +|(bu+av)-\rho|^2}   \approx 
{\dist_x(\zeta)\over |\zeta-\xi|^2}.
$$
So  the conclusion of the lemma is  also true in this first subcase of the fourth case.

\noindent {\bf Subcase:} $y\leq 0.$
    
    As $y=\tau(\xi),$  we see that  $\xi$ lies on the   ray  $\{bu+av=0\}\cap  \partial \Pi_x$ and $|(\hat u(y,v)+iv)-\xi|\leq c_2  |v|$   when $c_3$ is large enough.
    So  $|\xi|\approx |y|^{1/\gamma}\approx \dist^\star_x(\zeta)$  and $\dist_x(\zeta)\approx  v.$  Moreover,  $$ v+|(bu+av)-\rho|\approx  v+|bu+av|\approx |\xi|\approx  |\zeta-\xi| .$$
     This, combined with 
 assertion (5) of  Lemma 3.3 in \cite{NguyenVietAnh18a} and  \eqref{e:change-Poisson}  and  \eqref{e:dists-in-2D},  \eqref{e:conversion}, implies that
$$
P_x(\zeta,\xi)\approx {v\over v^2 +|(bu+av)-\rho|^2}   \approx 
{\dist_x(\zeta)\over |\zeta-\xi|^2}.
$$
So  the conclusion of the lemma is  also true in this last subcase of the fourth case.
\endproof

\begin{proposition}\label{P:estimate(3)}
   Suppose that $n\geq 2.$ Then    there are constants $c,\gamma>1$ which depend only on   $\lambda_1,\ldots,\lambda_n$  
 with the following property.  For every $x\in \X$   and for  $\zeta\in \Pi^{r_1}_x$ and $\xi\in\partial\Pi_x\cap \partial \H_{x,l},$   and    for every $1\leq  k\leq n$
 such that  $ \lambda_l/\lambda_k\not\in\R,$   we have 
 \begin{equation*}
    P_x(\zeta,\xi)\leq  c {   \min\big(\dist_{x,l}(\zeta),   \dist_{x,k}(\zeta)  \big)\over |\zeta-\xi|^2}  \,\cdot \Big({  \min\big( \max\big(\dist_{x,l}(\zeta),   \dist_{x,k}(\zeta)  \big),
    \dist_{x,k}(\xi)\big)
    \over \max \big(     \max\big(\dist_{x,l}(\zeta),   \dist_{x,k}(\zeta)  \big)             , \dist_{x,k}(\xi)\big) }\Big)^{\gamma-1}.
 \end{equation*}

\end{proposition}
\proof
Consider the sector $\Omega_2:=\H_{x,k}\cap \H_{x,l}.$ 
Applying  Proposition \ref{P:estimate(1)}  to $\Omega_1:=\Pi_x$ and $\Omega_2$ yields that
$$
P_x(\zeta,\xi)\leq P_{\Omega_2}(\zeta,\xi).
$$
Applying  Proposition  \ref{P:estimate(2)} to  the RHS and using \eqref{e:dist_x}, the result follows. See  Figure \ref{fig:M6} for an illustration of this proof.
\endproof



\begin{figure}
\centering
 
\begin{tikzpicture}[angle radius=2cm]
        \coordinate (A) at (1,0);
        \coordinate (B) at (3,0);
        \coordinate[label=above:$\xi$] (xi) at (1.5,0);
        \coordinate (C) at (0,1.7);
        \coordinate (D) at (0.5,0.5);
        \coordinate (w_0) at (3,5.5);
        \coordinate (w_1) at (2,5);
        \coordinate (w_2) at (1,4);
         \coordinate (w_k+1) at (4,0.25);
       
\draw (A)--(B)  node[
    pos=0.5,below] {$\partial\H_{x,l}$};
\draw (w_2)--(w_1);
\draw (C)--(w_2);
\draw[dashed] (C)--(D);
\draw (D)--(A);
\draw (A)--(B);
\draw (w_1)--(w_2) node[
    pos=0.5,sloped,above] {$\partial\H_{x,k}$};

    \coordinate (X) at (intersection cs:first line={(w_1)--(w_2)}, second line={(A)--(B)});
    \node at (X)  {$\bullet$}; 
         \draw[fill=orange!25]
         (w_0) -- (w_1) -- (w_2) -- (C) -- (D) -- (A) -- (B) --(w_k+1) ;
         \draw[fill=red]  (C)--(D);
\coordinate (Y) at (intersection cs:first line={(w_1)--(w_2)}, second line={(w_0)--(w_k+1)});
\coordinate (Z) at (intersection cs:first line={(w_0)--(w_k+1)}, second line={(A)--(B)});



      \draw (2,2.5)  node  [above]{\Large $\Omega_1:=\Pi_x$};

 
 \draw (2.8,1.5) node [above]{$\zeta$};
 
 \draw (2.8,1.5) node {$\bullet$};
 

 \draw (1.5,0) node [above]{$\xi$};
 
 \draw   (1.5,0)  node {$\bullet$};

        
\path
(X) edge [red, thick] (Y)
(X) edge [red, thick] (Z)
pic ["$\Omega_2$", draw, dashed,red] 
{angle=Z--X--Y};


\draw (w_1) node {$\bullet$};
\draw (w_2) node {$\bullet$};
\draw (C) node {$\bullet$};
\draw (D) node {$\bullet$};
\draw (A) node {$\bullet$};
\draw (B) node {$\bullet$};    
        
    \end{tikzpicture}

\caption{We apply  Proposition \ref{P:estimate(1)}  to $\Omega_1:=\Pi_x$ and  $\Omega_2:=\H_{x,k}\cap \H_{x,l},$
and to     $\zeta\in \Pi^{r_1}_x$ and $\xi\in\partial\Pi_x\cap \partial \H_{x,l}.$ } \label{fig:M6}
\end{figure}
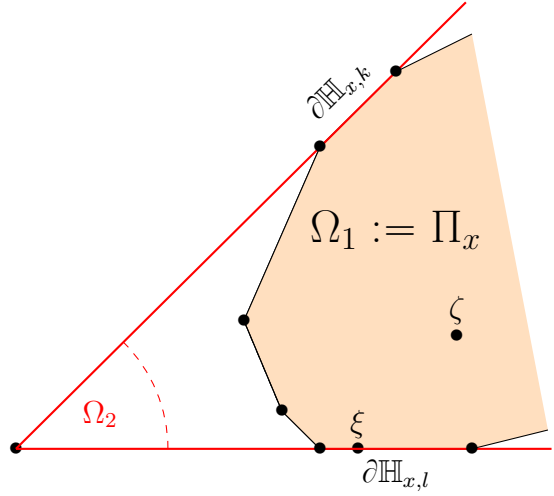

\section{Proof of the Main Theorem}\label{S:Main-Theorem}
Suppose without loss of generality that  $\widehat \Fc=(\overline\D^n, \widehat \Lc,\{0\})$  is a singular holomorphic  foliation  associated to a linear vector field  which is 
defined on an open neighborhood of $\overline\D^n.$ 
Suppose also  that
$\Fc$ is  the restriction of $\widehat \Fc$ on a  possibly smaller  open neighborhood of $\overline\D^n.$
Prior to the  proof of the Main Theorem, some auxiliary  estimates are needed.
\begin{lemma}\label{L:Poisson-inequ}
 There are constants $\gamma, c>1$  independent of $x\in\X'$ and $\zeta\in\Pi_x$ and $\xi\in\partial\Pi_x$ such that
\begin{equation*}
 P_x(\zeta,\xi)\leq  c \min\big(1, 
 \big({\dist^\star_x(\xi)\over \dist_x(\zeta)}\big)^{\gamma-1}\big) \dist_x(\zeta) {1\over |\zeta-\xi|^2}.
 \end{equation*}
\end{lemma}
\proof   
    Let $1\leq  l,k\leq n $ be   determined by  $\xi\in \partial \H_{x,l}$  and  $\dist_x(\zeta)=\dist_{x,k}(\zeta).$
   There are  three cases to consider.
   \\
   \noindent  {\bf  Case} $\lambda_l/\lambda_k\not\in\R:$
   Applying  Proposition \ref{P:estimate(3)} yields that
 \begin{equation*}
    P_x(\zeta,\xi)\leq  c {\dist_{x,k}(\zeta)\over |\zeta-\xi|^2}  \,\cdot  \Big({  \min\big( \max\big(\dist_{x,l}(\zeta),   \dist_{x,k}(\zeta)  \big),\dist_{x,k}(\xi)\big)
    \over \max \big(     \max\big(\dist_{x,l}(\zeta),   \dist_{x,k}(\zeta)  \big)             , \dist_{x,k}(\xi)\big) }\Big)^{\gamma-1}.
 \end{equation*}
 The  expression  in  big parenthesis in the RHS is smaller than or equal to  $\ { \dist_{x,k}(\xi)\over  \dist_{x,k}(\zeta)},$  which is,  in turn, bounded from  above by
 $ {\dist^\star_x(\xi)\over \dist_x(\zeta)}.  $ Hence,  the desired inequality follows.
\\
 \noindent  {\bf  Case} $l=k:$
   
Since  the  singularity $\{0\}$ is  weakly hyperbolic, there is $1\leq j  \leq n$ such that $\lambda_l/\lambda_j\not\in\R.$
We argue as  in the first case    for $j$  instead of $k,$ and the proof in this case is thereby completed. 
\\
\noindent  {\bf  Case} $l\not=k$ and $\lambda_l/\lambda_k\in\R:$

Let $\S$ be the strip  limited by two parallel lines  $\partial \Pi_{x,l} $ and $\partial \Pi_{x,k}.$
By  Proposition \ref{P:estimate(1)}  applied  to $\Omega_1:=\Pi_x$ and $\Omega_2:=\S,$ we get 
$$
P_x(\zeta,\xi)\leq P_\S(\zeta,\xi).
$$
Since  $\xi\in  \partial \H_l,$ and   $\partial \Pi_{x,l}, $ $\partial \Pi_{x,k}$  are parallel,
we see that $\dist^*_x(\xi)\geq  \dist(\partial \Pi_{x,l},\partial \Pi_{x,k}).$ On the other hand,
since  $\dist_x(\zeta)=\dist_{x,k}(\zeta)$  and the strip limited by two parallel lines $\partial \Pi_{x,l}, $ $\partial \Pi_{x,k}$
contains $\Pi_x,$  it follows that 
$$\dist_x(\zeta)=\dist_{x,k}(\zeta)\leq \dist_{x,l}(\zeta)=  \dist(\partial \Pi_{x,l},\partial \Pi_{x,k})-\dist_{x,k}(\zeta) ,$$
and hence
$$\dist_x(\zeta)\leq {1\over 2} \dist(\partial \Pi_{x,l},\partial \Pi_{x,k}) .$$
Combining  these  estimates, we  get 
   ${\dist^\star_x(\xi)\over \dist_x(\zeta)}\geq 2.$  
The desired conclusion follows from Proposition \ref{P:Widder} (2).
\endproof

The  following result  gives a precise  behavior of  $K_{s}(x,\xi)$  (introduced in  \eqref{e:K_s}) when the leaves  get close to the  hyperplanes $\{z_j=0\}.$
\begin{lemma}\label{L:K_s}
 There  is a constant $c,\gamma>1$ such that
for all $x\in\X',$  $s>1$ and $\xi\in\partial \Pi_x,$
$$
 K_{s}(x,\xi) \leq  c \min\big(1, 
 (\dist^*_x(\xi)/s)^{\gamma-1}\big).
 $$
\end{lemma}
\proof
Recall from Propositions \ref{P:Poisson-classic} and \ref{P:Poisson} that  $P_x(\cdot,\xi),$ for a fixed $\xi\in \partial \Pi_x,$ is  a positive  harmonic  function on $\Pi_x.$ Therefore, 
by
  Harnack's inequality, there is a constant  $c'>0$ independent of $x\in \X'$ and $\xi\in\partial \Pi_x$ such that
$$
 P_x(\zeta,\xi) \leq  c' P_x(\zeta',\xi)\quad\text{for}\quad  \zeta,\zeta'\in \Pi^{r_1}_x,\quad |\zeta-\zeta'|\leq 1.
$$  On the other hand,  for  every $\zeta\in \Pi^s_x$ there is  exactly one $j\in \N$ such that   $\zeta\in \Pi_x^{s+j}\setminus \Pi_x^{s+j+1},$
and for $\zeta\in \partial \Pi_x^{s+j}, $ we have $ \dist_x(\zeta)=s+j$ and hence  $ e^{2s-2\dist_x(\zeta)}=e^{-2j}  .$
Moreover,  there is a constant $c>0$ independent of $x$ and $s$ such that the  following   inequality holds
$$
\int_{\Pi^s_x} f(\zeta)d\Leb_2(\zeta)\leq c \sum_{j=0}^\infty  \int_{\zeta\in \partial \Pi_x^{s+j}} f(\zeta)d\Leb_1(\zeta)
$$
for every  positive function  $f$ defined on $\Pi^s_x$  satisfying
$$
f(\zeta) \leq  c' f(\zeta')\quad\text{for}\quad  \zeta,\zeta'\in \Pi^{s}_x,\quad |\zeta-\zeta'|\leq 1.
$$
Combining  all these inequalities,
we infer from \eqref{e:K_s} that
\begin{equation*}
 K_s(x,\xi)\lesssim  \sum_{j=0}^\infty  \int_{\zeta\in \partial \Pi_x^{s+j}}e^{-2j} P_x(\zeta,\xi)d\Leb_1(\zeta).
\end{equation*}
Applying   Lemma  \ref{L:Poisson-inequ}, we deduce  from the above estimate and from  the inequality
$\dist_x(\zeta)\leq s+j$ for $\zeta\in\partial \Pi^{s+j}_x$ that 
\begin{equation*}
 K_s(x,\xi)\leq c \min\big(1, 
 \big(\dist^\star_x(\xi)/s\big)^{\gamma-1}\big)  \cdot \sum_{j=0}^\infty  \int_{\zeta\in \partial \Pi_x^{s+j }}{e^{-2j} (s+j) \over |\zeta-\xi|^2}d\Leb_1(\zeta).
\end{equation*}
To  conclude the proof of the lemma, we only need to show that  the last sum is  uniformly bounded independently of $x\in\X.$ To this end
we will show that there is a constant $c>0$ independent 
of $x\in\X$ 
and $\xi\in\partial\Pi_x$  such that  
\begin{equation}\label{e:finiteness_xi} \int_{\zeta\in \partial \Pi_x^t}{1 \over |\zeta-\xi|^2}d\Leb_1(\zeta)\leq ct^{-1}.
 \end{equation}
Indeed, taking  for granted this inequality, we  apply it  to $t=s+j$ for $j\in\N$ and sum up the results.  This will imply that the above sum  is uniformly bounded.
To prove  \eqref{e:finiteness_xi}, we observe that  the edges of $ \partial \Pi_x^t$ are parallel to  those of $\Pi_x,$ and $\Pi_x^t$ possesses at most $n$ edges. 
Moreover,  $|\zeta-\xi|\geq t$ for $\zeta\in \partial \Pi_x^t.$
So the LHS of  \eqref{e:finiteness_xi} is bounded by $$\ n\cdot \sup_\ell\int_{\zeta\in \ell:\ |\zeta-\xi|\geq t}{1 \over |\zeta-\xi|^2}d\Leb_1(\zeta),$$
the supremum being taken aver all real lines $\ell \subset \C.$ A straightforward computation shows that  the above supremum is bounded by $O(t^{-1}).$
Hence, \eqref{e:finiteness_xi} follows, and  the  proof is  complete.
\endproof

\proof[End of the proof of Theorem \ref{T:main}]
 By  Lemma \ref{L:K_s},  the  family of functions
  $(K_s)_{s>0}:\ \K\to\R^+,$  
is  uniformly  bounded.  Moreover,  $\lim_{s\to\infty} K_s(x,\xi)=0$ for $(x,\xi)\in \K.$

On the other hand, consider  the  measure $\chi$  on $\K,$ given by
$$
\int_\K \alpha d\chi= \int_{x\in\X}\big(\int_{\xi\in\partial \Pi_x}\alpha(x,\xi)  h_x(\xi)d\Leb_1(\xi)\big) d\mu(x)
$$
for every continuous bounded test function $\alpha$ on $\K.$ 
By inequality \eqref{e:lower-bound-mass-T_x} in the proof of Lemma \ref{L:Poisson-representation},
 there is   a constant $c>0$ (independent of $x$) such that 
\begin{eqnarray*}
\|T\|_{\D^n\setminus (r_0\D)^n} &=&\int_{\D^n\setminus (r_0\D)^n}T\wedge (idz_1 \wedge  d\bar z_1 +\ldots+ idz_n \wedge d\bar z_n )
\\
&\geq&   \int_{x\in\X} \big(\int_{\Pi_x\setminus \Pi^{r_1}_x}   T_x \wedge (idz_1 \wedge  d\bar z_1 +\ldots+ idz_n \wedge d\bar z_n )\big)  d\mu(x)\\
&\geq&  c  \int_{x\in\X}\big(\int_{\Pi_x\setminus \Pi^{r_1}_x}h_x(\zeta)d\Leb_2(\zeta)\big) d\mu(x)\\
&\geq &c'c\int_{x\in\X}\big(\int_{\xi\in\partial \Pi_x}  h_x(\xi)d\Leb_1(\xi)\big) d\mu(x),
\end{eqnarray*}
where the last inequality holds for a constant $c'>0$ (independent of $x$) by an application of   Harnack's inequality  for positive harmonic  functions.
So $\chi$ is a  finite positive measure.
 
 Consequently,  we get, by  dominated convergence, that 
$
\lim_{s\to\infty}\int_\K K_sd\chi=0. 
$ 
 This, combined  with  Lemma \ref{L:Lelong-representation}, implies that
$$
0\leq \lim_{r\to 0+}  f(r)\leq n\cdot \lim_{s\to\infty}\int_\K K_sd\chi=  0,
$$
which,  coupled  with \eqref{e:F}-\eqref{e:f} and \eqref{e:Lelong_bisbis}, gives that
 $\nu ( T,0)=0,$ as  desired.
\endproof

\proof[Proof of Theorem  \ref{T:main_bis}]
 Using Remark
  \ref{R:stronger-version-f} and arguing as in the  proof of Theorem \ref{T:main} we  see that  $\hat F(r)<\infty$ for $0<r<1.$
Hence, the mass of $T$ on  $(r\D)^n$ is  finite  for all $r\in[0,1].$  This  proves assertion (1).

To prove  assertion (2) pick a  point $x\in\mathcal Z.$ There are two cases.

{\noindent \bf Case 1: $x=0.$} The proof of Theorem  \ref{T:main} also  work in this context and we get that $\nu(T,0)=0.$

{\noindent  \bf Case 2: $x\not=0.$}

Let $\U_x$ be a regular flow box of $\Fc$ which contains $x$ and which is away from $\{0\},$
Let $\T$ be a transversal $\U_x$ which   contains $x.$ 
By shrinking $\U_x$ if necessary, we may assume without loss of generality that $\T$ is a complex  manifold of dimension $n-1.$
Let $\widecheck \T\subset \T$ be the transversal of  $\widecheck \Fc$ on $\U_x.$   
Since $x\in \mathcal Z,$ it follows that $x\not\in\widecheck\T.$
By  Proposition \ref{P:decomposition}   we can write in $\U_x$ 
$$
T=\int_{t\in \widecheck \T} h_t[\B_t] d\nu(t),
$$
where, $\nu$ is a  positive Borel measure on $\widecheck \T,$ and  for $\nu$-almost  every $t\in \widecheck\T,$   $h_t$ denote the positive harmonic function associated to the current $T$ on the plaque  $\B_t. $
We  may assume  without loss of generality that $h_t(t)=1.$
By Harnack's inequality, there is a constant $c>0$ independent of $t$ such that
 $$  c^{-1}h_t(z)\leq  h_t(w)\leq    ch_t(z),\qquad  z,w\in \B_t.$$
 In particular, we get that  $h_t(z)\approx 1$ for $z\in \B_t$ as  $h_t(t)=1.$
 Using this and  the above  local description of $T$ on $\U_x$ and formulas \eqref{e:Lelong}--\eqref{e:Lelong_bis}--\eqref{e:Lelong_bisbis}, we  infer easily that $\nu(T,x)\leq c  \nu(\{x\}),$ where
 \begin{equation}\label{e:nu_x}
 \nu(\{x\}):=\lim_{\epsilon\to 0}  \nu(\{y\in\widecheck\T:\  \dist(y,x)<\epsilon\}),
 \end{equation}
 and $\dist(\cdot,\cdot)$ is a  distance induced  by a fixed Hermitian metric on the complex manifold $\T.$ 
 Therefore,  if one can show that $\nu(\{x\})=0,$ then it follows from the last inequality that   $\nu(T,x)=0,$ and we are done.
 So it  remains to prove that $\nu(\{x\})=0.$ 
 
 By assertion (1), the mass of $T$ on $\U_x$ is  finite.  This, combined with  the above description of $T$ on $\U_x$ 
 and the  above estimate  $h_t(z)\approx 1$ for $z\in \B_t,$ implies that  $\nu$ is a finite measure. 
Therefore, we deduce from  the equality $\bigcap_{j=1}^\infty \{y\in\widecheck\T:\  \dist(y,x)<j^{-1}\}=\varnothing$ 
and from \eqref{e:nu_x} that
 $\nu(\{x\})=0$ as  desired.
\endproof

\section{Proof of the global  result and concluding remarks} \label{S:Other-results}

First we recall  two results of  Forn{\ae}ss--Sibony
\cite{FornaessSibony05,FornaessSibony10}.
The proof given   here makes an emphasis on the  generalization of these results in  the higher dimension $n\geq 2.$
These results will be  needed in the  proof of Theorem \ref{T:main_global}.

In the  following  proposition  we are  concerned  with  directed positive harmonic  currents $T$  of  the  form
\begin{equation}\label{e:atomic-current} 
T=h [L_a] 
\end{equation}
 where $a$ is a point in $  X\setminus E$ and $h$ is a positive  harmonic function
on the leaf $L_a.$  
Let $f$ be the lifting of the harmonic function to the unit disc $\D$, that is, 
\begin{equation}\label{e:lifting_h} f = h \circ\varphi\qquad\text{on}\qquad \D,
\end{equation}
where $\varphi : \D \to L$ is a universal covering map. 

\begin{proposition}\label{P:FS}{\rm  (Forn{\ae}ss--Sibony \cite{FornaessSibony05,FornaessSibony10}) }
Let $\Fc=(X,\Fc,E)$    be  
a     singular  holomorphic  foliation   with      the set of  singularities $E$ in a    compact complex manifold  $X.$ 
Assume that
\begin{enumerate}
\item $E$  is a  finite set;
 \item   there is  no invariant analytic  curve;
\item  there is   no  non-constant holomorphic map $\C\to X$ such that
out of $E$   the image of $\C$ is locally contained in a leaf.
\end{enumerate}
Let $T$ be  a positive  harmonic  current  directed by $\Fc$ which has the form  \eqref{e:atomic-current}.
Then:
\begin{enumerate} 
\item [(i)]For every  neighborhood $U$ of $E,$  there is a constant $c_U>0$ such that  $h(x)\leq c_U$ for $x\in L_a\setminus U;$
 \item [(ii)] The function $f$ given in \eqref{e:lifting_h} has  nontangential limits $0$ $\Leb_1$-almost everywhere  on $\partial \D.$
\end{enumerate}

 \end{proposition}
 
\begin{remark}
 \rm The condition (3)  in Theorem  \ref{T:main_global}  is equivalent  to the Brody hyperbolicty of $\Fc$  in the sense of \cite{DinhNguyenSibony14}, that is,
 there is a constant $c>0$ such that  for every holomorphic map $\varphi$ from $\D$  to a  leaf, 
 \begin{equation}\label{e:Brody}
 |\varphi'(\zeta)|<{c\over  1-|\zeta|}\qquad\text{for}\qquad \zeta\in\D.
\end{equation}
 It is  my mistake  to   omit  this condition in my  previous  work \cite[Theorems  1.1 and 1.3]{NguyenVietAnh18a}.
\end{remark}
\proof 
Let $\U$ be  a flow box.
By Harnack's inequality, there is a constant $c>0$ which depends only on $\U$  such that
 $$  c^{-1}h(z)\leq  h(w)\leq    ch(z),\qquad  z,w\in P,$$
 where $P$ is a plaque of $\U$  which  is  contained  in  $L_a.$
 Since the mass of $T$ on $\U$ is   finite, we infer from  the  above inequality and   \eqref{e:atomic-current} that
if the leaf $L_a$ intersect $\U$   infinitely many times   in pairwise different plaques $P_j$  with $j\in \N,$  then the harmonic functions
$h_j := h|_{P_j}$ must go uniformly to zero as $j\to\infty.$
 Hence,  assertion (i) follows. 
   
   To  prove assertion (ii), consider  the set $S$ consisting of all $\zeta\in \partial \D$ of
  such that  $f$ has nontangential limits $f (\zeta )$   at $\zeta.$
 Since $h>0$ on $L_a,$ it follows that  the harmonic function  $f$ is  positive on $\D.$
 Hence $S$ is of full $\Leb_1$-measure on $\partial \D.$
    Consider  
$$S_0:= \left\{\zeta\in  S:\ f (\zeta) > 0\right\}.$$ 
Suppose in order to reach  a  contradiction that $\Leb_1(S_0)>0.$ 
We consider the curve $\{\varphi(re^{i\theta_0} ):\ r\in[0,1)\}.$ By the
above argument, it follows that this curve can only intersect finitely many plaques in
any flow box. 

Suppose that  some plaque $P$ is visited by this  curve infinitely
many times as $r \to  1.$ Note that $h$ must be constant on $P,$ hence constant
on the leaf $L_a,$ hence  $T$  is  a positive closed  current whose Lelong  numbers  at all points of  $L_a$ are $\geq 1.$
By Siu's theorem \cite{Siu},  $L_a$ is  contained  in a   proper analytic  set of $X.$  Since $X$ is  compact, 
$L_a$ is an analytic subset of $X\setminus E.$  As  $E$ is a finite set and $\dim_\C(X)\geq 2,$
we deduce from
 Remmert--Stein  theorem that
$L_a$ is  an invariant analytic  curve in $X.$ This  contradicts  assumption (2).

Consequently, we infer from the two previous paragraphs that the curve $\varphi(re^{i\theta_0} )$ converges as $r\nearrow 1$  to a singular point.
Since $E$ is a  finite set,  there are a  set $S_1\subset S_0$  with $\Leb_1(S_1)>0$  and a point $a_0\in E$ such that  
$\varphi(re^{i\theta} )\to  a_0$  as $r\nearrow 1$ for  all $\theta\in S_1.$

By assumption (3), $\Fc$ is  Brody  hyperbolic. So by \eqref{e:Brody}, there is a constant $c>0$ such that
\begin{equation}\label{e:Brody-consequence}
 |\varphi'(\zeta)|<{c\over  1-|\zeta|}\qquad\text{for}\qquad \zeta\in\D.
\end{equation}
Since  $X$ is  compact,  we infer from a theorem of  Lehto--Virtanen \cite{LehtoVirtanen} that   $\varphi$ is  bounded  in any angle with  vertex $\zeta$ for every  $\zeta\in S_1.$ 
Using the argument  in Privalov  \cite{CollingwoodLohwater}, we  can construct a subset  $S_2\subset S_1$   with $\Leb_1(S_2)>0$  and a Jordan subdomain $G\subset \D$  with rectifiable boundary  
such that  $ S_2\subset \partial G,$
and that $G$  contains an angle  with vertex at  $\zeta$  for all $\zeta\in S_2,$  and  that $\varphi(G)$ is contained in a local chart around $a_0.$
By Lindel\"of's theorem, $\varphi$ has nontangential limits $a_0$ on  $S_2.$
By Privalov's  theorem applied  to $\varphi|_G,$  we see that $\varphi\equiv a_0.$ Hence, $\varphi\equiv a_0$ on $\D,$ which  contradicts  $\varphi(\D)=L_a.$
So  $\Leb_1(S_1)=0,$  and hence  $\Leb_1(S_0)=0.$  This proves assertion (ii).

\endproof

\endproof

\begin{theorem}\label{T:FS} {\rm  (Forn{\ae}ss--Sibony \cite[Corollary 2]{FornaessSibony10}) }Let $\Fc=(X,\Lc,E)$ be a  singular holomorphic foliation  as in  Proposition \ref{P:FS}.
Assume in addition that  $\dim X=2.$
Then every  positive  harmonic current directed by $\Fc$ is  diffuse. 
\end{theorem}
\proof
Assume in order  to get  a  contradiction that  there is a   positive  harmonic current $T$ directed by $\Fc$   which is not diffuse. So $T$  has an atomic part, i.e. a Dirac mass at   a point
$a\in X\setminus E.$ The restriction $T$ to the leaf $L_a$ is a non-zero positive harmonic current. We can
normalize so that the transverse measure  is the  Dirac mass at $a.$ Then we have a positive harmonic
function $h$ defined on $L_a.$

By Proposition  \ref{P:FS} (ii) (see e.g. \cite[Corollaries 6.15 and  6.44]{ABR}), there is positive Borel  measure $\nu$ on $\partial  \D$   with  support $S$  such that  $\Leb_1(S)=0$  and 
$$
f(\zeta)=\int_{\partial \D} P_\D(\zeta,\xi)d\nu(\xi),
$$
where $P_D$ is  the Poisson kernel of  $\D.$
The function $f$ should be
 unbounded, since  otherwise  $\nu$  would have a  bounded density with respect to $d\Leb_1,$  which  would  contradict  that $\supp(\nu)=S$ and $\Leb_1(S)=0.$

On the other hand, let $a_0\in E$ be a singular point.  Fix  a   local holomorphic coordinates on an open neighborhood $\U$ of  $a_0$ on which 
$\Fc$ is  identified with a  local model   $(\D^n,\Lc, 0)$ of the form  \eqref{e:local-model}.
Let $\X$  be  given by  Lemma  \ref{L:transversal}.  By Proposition \ref{P:decomposition} (2),  for  every $x\in \X,$ there is a constant $c_x>0$ such that  $h=c_xh_x$ on  $L_x.$
Consequently, by  Remark  \ref{R:dim2} (that is, by Lemma \ref{L:Poisson-representation}  for $n=2$),  we have 
$$
h(\zeta)=\int_{\partial \Pi_x} P_x(\zeta,\xi)h(\xi)d\Leb_1(\xi).
$$
On the other hand, by Proposition  \ref{P:FS} (i) and by the inclusion $\X\subset\partial\D^n,$
$h$ is uniformly bounded    on  $\partial\Pi_x$ independently of $x\in \X.$
So  by the above integral  representation,  $h$ is also  
uniformly bounded on $\Pi_x$. Hence, $h$ must be uniformly bounded on a neighborhood of each  singular point $a_0\in E.$
This, combined with  Proposition  \ref{P:FS} (i), implies that $h$  must be uniformly bounded on  $L_a.$
This contradicts  the unboundedness of $f.$

Hence, $T$ is  diffuse.

\endproof

\proof[End of the proof of Theorem  \ref{T:main_global}]
To prove  that $T$ is  diffuse, we argue as in the proof of Theorem \ref{T:FS}  using  Lemma \ref{L:Poisson-representation} for all $n\geq 2.$

Now  we prove that  $\nu (T,x)=0$ for all $x\in X.$ Let $x\in X.$ 
Consider two cases.
\\
\noindent  {\bf Case 1:} $x\not\in E.$

Let $\U\simeq \B\times\T$ be a regular flow box with transversal $\T$ which   contains $x.$
By Proposition \ref{P:decomposition} (1), 
we can write in $\U$ 
$$
T=\int h_t[\B_t] d\mu(t),
$$
where $\mu$  is a positive Radon measure  on $\T,$ and for $\mu$-almost  every $t\in \T,$   $h_t$  is a positive harmonic function 
on the plaque  $\B_t\simeq \B\times\{t\}. $ By Harnack's inequality, there is a constant $c>0$ independent of $t$ such that
 $$  c^{-1}h_t(z)\leq  h_t(w)\leq    ch_t(z),\qquad  z,w\in \B_t.$$
 Using this and  the above  local description of $T$ on $\U$ and formula (\ref{e:Lelong_bisbis}), we  infer easily a constant
 $c>0$  depending only on $\U$  such that $\nu(T,x)\leq c  \mu(\{x\}).$
 On the  other hand, since  we have shown that $T$ is  diffuse, $\mu(\{x\})=0.$ Hence,   $\nu(T,x)=0.$
 
\noindent  {\bf Case 2:} $x\in E.$

  Fix a (local) holomorphic coordinates system of $X$ on a singular  flow box $\U_x$ of $x$  such that
$(\U_x,x)$ is identified with  $(\D^n,0)$ and 
 the
leaves of $\Fc$  on this  box  are integral curves of the  linear vector field
$\Phi$ given by  \eqref{e:local-model}.
Consider 
$$
J:=\left\lbrace  j:\  1\leq j\leq n \quad\text{and} \quad T \quad \text{gives mass to the invariant hyperplane}\quad \{z_j=0\}  \right\rbrace.
$$
If $J=\varnothing,$
then we  are able to apply   Theorem \ref{T:main}  which gives    $\nu(T,x)=0$ as  desired.

Consider the case $J\not=\varnothing.$ Let $T_j$ be the restriction of $T$ on the invariant hyperplane $\{z_j=0\}.$
This is a directed positive  harmonic current. 
Consider $T':=T-\sum_{j\in J} T_j.$ So $T'$ is  also a directed positive harmonic current giving no mass to any coordinate invariant hyperplane $\{z_j=0\},$ and  we have by \eqref{e:Lelong_bis}--\eqref{e:Lelong_bisbis},
$$
\nu(T,x)\leq \nu(T',x)+ \sum_{j\in J}\nu( T_j,x)
.$$
By Theorem \ref{T:main},  $\nu(T',0)=0.$
Therefore,   in order to prove that   $\nu(T,x)=0,$ we only need to show that     $\nu(T_j,0)=0$ for $j\in J.$
Observe that since $0\in\C^n$ is a  hyperbolic singulatity, the restriction  of  $\Fc$ on $\{z_j=0\}$  admits $0\in\C^{n-1}$ as  a hyperbolic  singularity.
We  can argue as above by going down in one dimension by restricting  $\Fc$ and $T$  to the  invariant hyperplane
$\{z_j=0\}.$  We repeat this procedure. It should   stop  after a finite  steps.
Otherwise,  we  would go to a  plane  $H$ of   dimension $n=2.$   Then, the two invariant hyperplanes (i.e. two  separatrices in this context)   of $\Fc|_H$ are   two leaves of $\Fc,$   each  one of these  leaves is of the form
$$ \left\lbrace z=(z_1,\ldots,z_n)\in\D^n:\ z_l=0\qquad\text{for all}\qquad  l\not=j  \right\rbrace \quad\text{for some}\quad 1\leq j\leq n.   $$
$T|_H$  cannot give mass to none  of them, otherwise $T$ would give mass to a leaf, which in turn implies that this  leaf is an invariant  analytic curve,  which is  impossible by assumption (1).
 Consequently, we  are able to apply   Theorem \ref{T:main} to get that     $\nu(T|_H,0)=0.$ This completes the proof.  
\endproof

\proof[Proof of Corollary \ref{C:main_2}]
By Brunella  \cite{Brunella},  if
   all  the singularities of a  foliation $\Fc\in \Fc_d(\P^k)$ are    hyperbolic and  $\Fc$ does not possess any invariant algebraic  curve, 
   then $\Fc$ admits no  nontrivial  directed positive closed current. In particular, assumption (3) of Theorem \ref{T:main_global}  is  fulfilled. Clearly,
    all two other  assumptions  of this theorem are also fulfilled.
   This  theorem  implies the  corollary.

Let $\Fc_d (\P^n)$ be  the space of all singular holomorphic folitions of degree $d$ in $\P^n.$  
By Jouanolou \cite{Jouanolou} and Lins Neto-Soares \cite{NetoSoares}, there is a real Zariski dense open set $\mathcal H(d)\subset  \Fc_d(\P^k)$ such that for every $\Fc\in \mathcal H(d),$
  all  the singularities of $\Fc$ are    hyperbolic and  $\Fc$ does not possess any invariant algebraic  curve. So  a generic foliation in $\Fc_d (\P^n)$
  satisfies  the  assumptions of Corollary \ref{C:main_2}.
\endproof

We conclude the article with a remark and an open question.
\begin{remark}\rm
By Dinh--Wu  \cite[Theorem 1.1]{DinhWu}, our main result (Theorem  \ref{T:main}) is  essentially sharp.

When the singularities  are  linearizable but not    weakly hyperbolic, the study  of Lelong numbers seems  difficult.    Chen's recent article \cite{Chen} gives a partial result in this  direction for  dimension $n=2.$
It seems  to be interesting to find  sufficient conditions  on the nature of the singularity  $\{0\}$  to ensure that  the Lelong number  of $T$ at the origin  is  zero.
\end{remark}


\end{document}